\documentclass[a4paper]{amsart}

\usepackage[all]{xy}
\usepackage{amsmath}
\usepackage{amssymb}
\usepackage{amsthm}
\usepackage{latexsym}
\usepackage{enumerate}
\usepackage{prettyref}
\usepackage{hyperref}
\usepackage{bm}
\usepackage{mathrsfs}

\newtheorem{theorem}{Theorem}[section]
\newrefformat{thm}{\hyperref[{#1}]{Theorem~\ref*{#1}}}

\newrefformat{def}{\hyperref[{#1}]{Definition~\ref*{#1}}}
\newtheorem{lemma}[theorem]{Lemma}
\newrefformat{lem}{\hyperref[{#1}]{Lemma~\ref*{#1}}}
\newtheorem{proposition}[theorem]{Proposition}
\newrefformat{prop}{\hyperref[{#1}]{Proposition~\ref*{#1}}}
\newtheorem{corollary}[theorem]{Corollary}
\newrefformat{cor}{\hyperref[{#1}]{Corollary~\ref*{#1}}}
\newtheorem{remark}[theorem]{Remark}
\newrefformat{rem}{\hyperref[{#1}]{Remark~\ref*{#1}}}
{
\newtheorem{examplecore}[theorem]{Example}}
\newrefformat{ex}{\hyperref[{#1}]{Example~\ref*{#1}}}

\newenvironment{proofof}[1]{\vspace{.2cm}\noindent\textsc{Proof of
    #1:}}{\hspace*{\fill} $\blacksquare$\par\vspace{.1cm}}

\newcommand{\op}{\operatorname}

\begin{document}

\title{Variations in \texorpdfstring{$\mathbb{A}^1$}{A1} on a theme of Mohan Kumar}  

\author{Matthias Wendt}

\date{September 2018}

\address{Matthias Wendt, Institut f\"ur Mathematik, Universit\"at Osnabr\"uck, Albrechtstra\ss{}e 28a, 49076 Osnabr\"uck, Germany}
\email{m.wendt.c@gmail.com}


\begin{abstract}
For every prime $p$, Mohan Kumar constructed examples of stably free modules of rank $p$ on suitable $p+1$-dimensional smooth affine varieties. This note discusses how to detect the corresponding unimodular rows by an explicit motivic cohomology group. Using the recent developments in the $\mathbb{A}^1$-obstruction classification of vector bundles, this provides an alternative proof of non-triviality of Mohan Kumar's stably free modules. The reinterpretation of Mohan Kumar's examples also allows to produce interesting examples of stably trivial torsors for other algebraic groups.
\end{abstract}

\maketitle
\setcounter{tocdepth}{1}
\tableofcontents

\section{Introduction}

The starting point of  this note is the following theorem of Mohan Kumar \cite{mohan:kumar} which provides important examples of stably free modules of high rank:

\begin{theorem}[Mohan Kumar]
\label{thm:mk1}
Let $k$ be an algebraically closed field. For every prime $p$, there exists   a $(p+1)$-dimensional smooth affine variety $X=\op{Spec}A$ over $k(T)$ and a nontrivial stably free $A$-module of rank $p$, given by a unimodular row of length $p+1$.
\end{theorem}

The main goal of the present note is to reinterpret the geometric constructions underlying the results in \cite{mohan:kumar}; here is a short outline of the basic ideas: The unimodular row defining the stably free module can be viewed as a morphism $X\to\mathbb{A}^{p+1}\setminus\{0\}$. In the setting of $\mathbb{A}^1$-homotopy theory, we have a composition of maps
\[
[X,\mathbb{A}^{p+1}\setminus\{0\}]_{\mathbb{A}^1}\to \op{H}^p_{\op{Nis}}(X,\mathbf{K}^{\op{MW}}_{p+1})\to \op{H}^p_{\op{Nis}}(X,\mathbf{K}^{\op{M}}_{p+1}/p!). 
\]
Some information on the relevant homotopy sheaves is provided in Section~\ref{sec:prelims}. For now, we note that the first non-trivial $\mathbb{A}^1$-homotopy sheaf of $\mathbb{A}^{p+1}\setminus\{0\}$ is $\bm{\pi}_p^{\mathbb{A}^1}(\mathbb{A}^{p+1}\setminus\{0\})\cong\mathbf{K}^{\op{MW}}_{p+1}$, and the first map in the above composition is induced from the corresponding morphism $\mathbb{A}^{p+1}\setminus\{0\}\to \op{K}(\mathbf{K}^{\op{MW}}_{p+1},p)$ which is the first nontrivial stage of the Postnikov tower for $\mathbb{A}^{p+1}\setminus\{0\}$. The second map is induced from the composition $\mathbf{K}^{\op{MW}}_{p+1}\to \mathbf{K}^{\op{M}}_{p+1}\to\mathbf{K}^{\op{M}}_{p+1}/p!$ of the quotient by $\eta$ and the quotient by $p!$. Now the variety $X$ constructed by Mohan Kumar arises from a Zariski covering of another variety $X'$, and the image of the unimodular row in $\op{H}^p_{\op{Nis}}(X,\mathbf{K}^{\op{M}}_{p+1}/p!)$ can be checked to be non-trivial because it has non-trivial image under the connecting map $\op{H}^p_{\op{Nis}}(X,\mathbf{K}^{\op{M}}_{p+1}/p!)\to\op{CH}^{p+1}(X')/p!$ of the Mayer--Vietoris sequence. By the recent computations due to Asok and Fasel in \cite{AsokFaselSpheres}, we know  that there is a natural map 
\[
\op{H}^p_{\op{Nis}}(X,\mathbf{K}^{\op{M}}_{p+1}/p!)\to \op{H}^p_{\op{Nis}}(X,\bm{\pi}^{\mathbb{A}^1}_p({\op{B}}\op{SL}_p)),
\]
and some exact sequence chasing plus the above information can be used to show that the lifting class of the composition $X\to\mathbb{A}^{p+1}\setminus\{0\}\to{\op{B}}\op{SL}_p$ yields a non-trivial class in $\op{H}^p_{\op{Nis}}(X,\bm{\pi}^{\mathbb{A}^1}_p({\op{B}}\op{SL}_p))$. This analysis provides an alternative $\mathbb{A}^1$-homotopical proof of the non-triviality of the stably free modules constructed by Mohan Kumar, cf.~Theorem~\ref{thm:mkalt} and the discussion in Section~\ref{sec:comparison}.

The relevance of this reformulation comes from the recent approach to torsor classification via $\mathbb{A}^1$-homotopy theory: the representability results \cite{gbundles,gbundles2} together with obstruction-theoretic methods \cite{MField,AsokFasel} relate the classification of $G$-torsors over smooth affine varieties to Nisnevich cohomology with coefficients in $\mathbb{A}^1$-homotopy sheaves of the classifying space of $G$. The above reinterpretation explains exactly how Mohan Kumar's stably free modules fit into the $\mathbb{A}^1$-topological approach to vector bundle classification. Moreover, there are several other classifying spaces of algebraic groups where quotients of Milnor K-theory appear in $\mathbb{A}^1$-homotopy sheaves. The above $\mathbb{A}^1$-topological construction of stably free modules can then be adapted to provide stably trivial torsors for other algebraic groups. For example, a second class of examples of stably trivial torsors, for the symplectic groups, follows from the vector bundle case by some further exact sequence chasing, cf.~Theorem~\ref{thm:sp}: 

\begin{theorem}
Let $k$ be an algebraically closed field of characteristic $\neq 2$. For every odd prime $p$, there exists a $p+1$-dimensional smooth affine variety over $k(T)$ and a stably trivial non-trivial $\op{Sp}_{p-1}$-torsor over it. Clearing denominators, there exists a $p+2$-dimensional smooth affine variety over $k$ with a stably trivial non-trivial $\op{Sp}_{p-1}$-torsor over it.
\end{theorem}

Taking the underlying projective module and adding a trivial line recovers the stably free modules of Theorem~\ref{thm:mk1}, i.e., Mohan Kumar's stably free modules have (stably trivial) symplectic lifts.

Finally, the original motivation for looking at Mohan Kumar's example through the eyes of $\mathbb{A}^1$-homotopy theory comes from the case $p=3$. There is also a Milnor K-cohomology group relevant for the classification of octonion algebras with hyperbolic norm form. Applying the Zorn construction to the examples of Mohan Kumar provides examples of such algebras over schemes of minimum possible dimension, cf. the discussion in Section~\ref{sec:rest} and \cite{octonion}. 

\subsection{Acknowledgements} 
This note was written during a pleasant stay at Institut Mittag-Leffler, in the program ``Algebro-geometric and homotopical methods'', and I thank the institute for its hospitality. I would like to thank Chuck Weibel for comments on an earlier version, Jean Fasel for pointing out some mistakes in the arguments for triviality of connecting morphisms and Aravind Asok for very helpful suggestions related to obstruction theory issues. Comments of two anonymous referees helped to improve the presentation; I am particularly grateful for comments requesting to include full details for the computation of the cohomology operations in Section~\ref{sec:a1homotopy} which I hope led to clarification of the proof and exhibited a minor 2-torsion mistake in the previous version.

\section{Recollection on \texorpdfstring{$\mathbb{A}^1$}{A1}-homotopy and representability} 
\label{sec:prelims}

In this section, we give a short recollection on the relevant input we use from $\mathbb{A}^1$-homotopy theory. 

We assume the reader is familiar with the basic definitions of $\mathbb{A}^1$-homotopy theory, cf.~\cite{MField}. Short introductions to those aspects relevant for the obstruction-theoretic torsor classification can be found in papers of Asok and Fasel, cf. e.g. \cite{AsokFasel,AsokFaselSplitting}. The notation in the paper generally follows the one from \cite{AsokFasel}. We generally assume that we are working over base fields of characteristic $\neq 2$. 

When considering $\mathbb{A}^1$-homotopy sheaves, only the simplicial grading will appear, i.e., for a pointed $\mathbb{A}^1$-connected space $(\mathscr{X},x)$, we will denote by $\bm{\pi}^{\mathbb{A}^1}_n\mathscr{X}$ the Nisnevich sheafification of the presheaf $U\mapsto [\op{S}^n_{\op{s}}\wedge U_+,(\mathscr{X},x)]$. 

\subsection{Representability theorem and obstruction theory}
\label{sec:representable}

Now we recall the representability theorem for torsors and introduce some notation for Postnikov towers which are used for the obstruction-theoretic approach to torsor classification in $\mathbb{A}^1$-homotopy.

The following representability theorem, significantly generalizing an earlier result of Morel in \cite{MField}, cf. also \cite[Theorem 6.22]{schlichting}, has been proved in \cite{gbundles,gbundles2}. 

\begin{theorem}
\label{thm:representability}
Let $k$ be an infinite field, and let $X=\op{Spec} A$ be a smooth affine $k$-scheme. Let $G$ be a reductive group such that each absolutely almost simple component of $G$ is isotropic. Then there is a bijection 
\[
\op{H}^1_{\op{Nis}}(X;G)\cong [X, {\op{B}}_{\op{Nis}}G]_{\mathbb{A}^1}
\]
between the pointed set of isomorphism classes of rationally trivial $G$-torsors over $X$ and the pointed set of $\mathbb{A}^1$-homotopy classes of maps $X\to {\op{B}}_{\op{Nis}}G$.
\end{theorem}

Here, ${\op{B}}_{\op{Nis}}G$ refers to the classifying space of rationally trivial $G$-torsors for an algebraic group $G$ which can be constructed by the usual simplicial bar construction. Classical results on the Grothendieck--Serre conjecture state that for isotropic reductive groups $G$ rationally trivial torsors are automatically Zariski-locally trivial. In the case when the group $G$ is special, i.e., when all torsors are Zariski-locally trivial, the index $\op{Nis}$ will be omitted and the classifying space will simply be denoted by ${\op{B}}G$. The groups that will be discussed will be $G=\op{SL}_n$ (for most part of the paper) as well as the symplectic groups $G=\op{Sp}_{2n}$, the spin groups $\op{Spin}(n)$ and exceptional group $\op{G}_2$ (in Section~\ref{sec:rest}). 

The representability theorem translates questions about $G$-torsor classification into questions about $\mathbb{A}^1$-homotopy classes of maps into the classifying space. In particular, we can prove that a torsor is non-trivial by exhibiting some $\mathbb{A}^1$-topological invariant showing that its classifying map is not null-homotopic. On the other hand, we can deduce the existence of torsors over smooth affine schemes with suitable properties by producing maps $X\to{\op{B}}_{\op{Nis}}G$, e.g. by obstruction-theoretic methods.

While the study of $\mathbb{A}^1$-homotopy classes maps into classifying spaces may not seem an easier subject than the torsor classification, the other relevant tool actually allowing to prove some meaningful statements (by analyzing maps into the classifying space) is obstruction theory. The basic statements concerning obstruction theory as applied to torsor classification can be found in various sources, such as \cite{MField} or \cite{AsokFasel,AsokFaselSplitting}. We only give a short list of the relevant statements which are enough for our purposes. 

The point of obstruction theory is to use the Postnikov tower to analyze pointed maps $[X,Y]_{\mathbb{A}^1,\bullet}$. Note that the representability theorem above related the isomorphism classes of $G$-torsors to \emph{unpointed} (or free) maps $[X,{\op{B}}_{\op{Nis}}G]_{\mathbb{A}^1}$. For the groups we consider in this paper, the classifying spaces ${\op{B}}_{\op{Nis}}G$ are $\mathbb{A}^1$-simply connected, hence the natural map from pointed to unpointed homotopy classes of maps is a bijection and the obstruction theory can be directly used to analyze $[X,{\op{B}}_{\op{Nis}}G]_{\mathbb{A}^1}$.

Let $(\mathscr{Y},y)$ be a pointed, 
 $\mathbb{A}^1$-simply connected space. Then there is a sequence of pointed 
$\mathbb{A}^1$-simply connected spaces, the Postnikov sections $(\tau_{\leq i}\mathscr{Y},y)$, with morphisms $p_i\colon\mathscr{Y}\to\tau_{\leq i}\mathscr{Y}$ and morphisms $f_i\colon\tau_{\leq i+1}\mathscr{Y}\to\tau_{\leq i}\mathscr{Y}$ such that
\begin{enumerate}
\item $\bm{\pi}_j^{\mathbb{A}^1}(\tau_{\leq i}\mathscr{Y})=0$ for $j>i$,
\item the morphism $p_i$ induces an isomorphism on $\mathbb{A}^1$-homotopy sheaves in degrees $\leq i$,
\item the morphism $f_i$ is an $\mathbb{A}^1$-fibration, and the $\mathbb{A}^1$-homotopy fiber of $f_i$ is an Eilenberg--Mac~Lane space of the form $\op{K}(\bm{\pi}_{i+1}^{\mathbb{A}^1}(\mathscr{Y}),i+1)$,
\item the induced morphism $\mathscr{Y}\to\tau_{\leq i}\op{holim}_i\mathscr{Y}$ is an $\mathbb{A}^1$-weak equivalence. 
\end{enumerate}
Moreover, $f_i$ is a principal $\mathbb{A}^1$-fibration, i.e., there is a morphism, unique up to $\mathbb{A}^1$-homotopy, 
$$
k_{i+1}\colon\tau_{\leq i}\mathscr{Y} \to\op{K}(\bm{\pi}_{i+1}^{\mathbb{A}^1}(\mathscr{Y}),i+2)
$$
called the $i+1$-th \emph{$k$-invariant} and an $\mathbb{A}^1$-fiber sequence
$$
\tau_{\leq i+1}\mathscr{Y} \to \tau_{\leq i}\mathscr{Y} \xrightarrow{k_{i+1}}\op{K}(\bm{\pi}_{i+1}^{\mathbb{A}^1}(\mathscr{Y}),i+2). 
$$ 

From these statements, one gets the following consequence: for a smooth $k$-scheme $X$ and a pointed, 
 $\mathbb{A}^1$-simply connected space  $\mathscr{Y}$, a given pointed map $g^{(i)}\colon X_+\to\tau_{\leq i}\mathscr{Y}$ lifts to a map $g^{(i+1)}\colon X_+\to\tau_{\leq i+1}\mathscr{Y}$ if and only if the following composite is null-homotopic:
$$
X_+\xrightarrow{g^{(i)}}\tau_{\leq i}\mathscr{Y}\to\op{K}(\bm{\pi}_{i+1}^{\mathbb{A}^1}(\mathscr{Y}),i+2), 
$$
or equivalently, if the corresponding obstruction class vanishes in the cohomology group $\op{H}^{i+2}_{\op{Nis}}(X;\bm{\pi}_{i+1}^{\mathbb{A}^1}(\mathscr{Y}))$. If this happens, then the possible lifts are parametrized by the quotient of the following set of homotopy classes of maps 
$$
[X_+,\op{K}(\bm{\pi}_{i+1}^{\mathbb{A}^1}(\mathscr{Y}),i+1)]_{\mathbb{A}^1}\cong \op{H}^{i+1}_{\op{Nis}}(X;\bm{\pi}_{i+1}^{\mathbb{A}^1}(\mathscr{Y}))
$$
modulo the standard action of $[X_+,\Omega\tau_{\leq i}\mathscr{Y}]_{\mathbb{A}^1}$. This quotient (or sometimes the cohomology group $\op{H}^{i+1}_{\op{Nis}}(X;\bm{\pi}_{i+1}^{\mathbb{A}^1}(\mathscr{Y}))$) will usually be called the \emph{lifting set}. Again, as noted above, since $\mathscr{Y}$ is $\mathbb{A}^1$-simply connected, there is no difference between considering pointed maps $X_+\to\tau_{\leq i}\mathscr{Y}$ or unpointed maps $X\to\tau_{\leq i}\mathscr{Y}$.

We want to state clearly what this means for the torsor classification over smooth schemes. If we have a torsor, then the map into the classifying space ${\op{B}}_{\op{Nis}}G$ associated by the representability theorem~\ref{thm:representability} is completely described by a sequence of classes in the lifting sets $\op{H}^{i+1}_{\op{Nis}}(X;\bm{\pi}_{i+1}^{\mathbb{A}^1}({\op{B}}_{\op{Nis}}G))$, well-defined up to the respective action of $[X_+,\Omega\tau_{\leq i}{\op{B}}_{\op{Nis}}G]_{\mathbb{A}^1}$. Only indices $0\leq i+1\leq n$ can appear for schemes of dimension $n$ since the Nisnevich cohomological dimension equals the Krull dimension. Conversely, to construct a torsor, one needs a sequence of lifting classes as above, such that the associated obstruction classes in the groups $\op{H}^{i+2}_{\op{Nis}}(X;\bm{\pi}_{i+1}^{\mathbb{A}^1}({\op{B}}_{\op{Nis}}G))$ vanish. Put bluntly, $\mathbb{A}^1$-obstruction theory translates questions about $\mathbb{A}^1$-homotopy classes of maps (from smooth schemes) into cohomological terms.

\subsection{Cohomology of strictly $\mathbb{A}^1$-invariant sheaves}
\label{sec:rscomplex}

We now provide a short recollection on strictly $\mathbb{A}^1$-invariant sheaves and their Nisnevich cohomology. All of the material here can be found in \cite{MField}. Recall that a sheaf $\mathbf{A}$ of abelian groups on $\op{Sm}/k$ is called \emph{strictly $\mathbb{A}^1$-invariant} if for each smooth $k$-scheme $X$ the map  $\op{pr}_1^\ast\colon\op{H}^i_{\op{Nis}}(X,\mathbf{A})\to\op{H}^i_{\op{Nis}}(X\times\mathbb{A}^1,\mathbf{A})$ is an isomorphism. 

As a matter of notation, strictly $\mathbb{A}^1$-invariant sheaves will usually be denoted by bold letters, following notational conventions from \cite{AsokFasel}. Most of the sheaves we will consider in this paper will be K-theory sheaves of various sorts where the superscript will usually indicate which type of K-theory is referred to: $\mathbf{K}^{\op{Q}}_i$ is (the Nisnevich sheafification of) Quillen's algebraic K-theory, $\mathbf{K}^{\op{M}}_i$ denotes Milnor K-theory. Morel's Milnor--Witt K-theory sheaves will be denoted by $\mathbf{K}^{\op{MW}}_i$, the definition can be found in \cite[Section 2]{MField}. As a special case, $\mathbf{K}^{\op{MW}}_0$ is the sheaf of Grothendieck--Witt rings, which is also denoted by $\mathbf{GW}$. For the discussion in Section~\ref{sec:mk2}, we'll also need the sheaves $\mathbf{I}^n$ which are the Nisnevich sheafifications of the presheaves of $n$-th powers of the fundamental ideal in the Witt ring. 

If $k$ is a perfect field, $K/k$ is any extension and $\mathscr{X}$ is an $\mathbb{A}^1$-simply-connected simplicial presheaf over $\op{Sm}/K$, then its $\mathbb{A}^1$-homotopy sheaves $\bm{\pi}^{\mathbb{A}^1}_n(\mathscr{X})$ are strictly $\mathbb{A}^1$-invariant sheaves of groups. In the special case where $K=k$ is an infinite perfect field, this is one of the main results of Morel's theory in \cite{MField}. A version of Gabber's presentation lemma over finite fields was established in \cite{hogadi:kulkarni}, and this allows to drop the requirement of infinity for the base field $k$. The general case which allows base change along an extension $K/k$ to a possibly non-perfect field follows from \cite[Lemma A.2, A.4]{hoyois}. This implies, in particular, that the $\mathbb{A}^1$-homotopy sheaves of ${\op{B}}\op{SL}_n$ over rational function fields $k(T)$ of positive characteristic are strictly $\mathbb{A}^1$-invariant because ${\op{B}}\op{SL}_n$ is defined over the (perfect) prime field. 

If $\mathbf{A}$ is a strictly $\mathbb{A}^1$-invariant sheaf of abelian groups on $\op{Sm}/k$, then for any smooth $k$-scheme $U$ the unit $1\in\mathbb{G}_{\op{m}}$ defines a morphism $1\times\op{id}\colon U\to\mathbb{G}_{\op{m}}\times U$, and the \emph{contraction} of $\mathbf{A}$ is defined to be the sheaf
\[
\mathbf{A}_{-1}(U)=\ker\left(\mathbf{A}(\mathbb{G}_{\op{m}}\times U)\xrightarrow{(1\times\op{id})^\ast}\mathbf{A}(U)\right). 
\]
This construction can be iterated to yield $\mathbf{A}_{-n}$ for $n\in\mathbb{N}$; it is an exact functor on the category of strictly $\mathbb{A}^1$-invariant sheaves on $\op{Sm}/k$, cf. \cite[Lemmas 2.32 and 7.33]{MField}. 

If $\mathbf{A}$ is a strictly $\mathbb{A}^1$-invariant sheaf of abelian groups, there is a natural $\mathbb{G}_{\op{m}}$-action on the contractions of $\mathbf{A}$, defined as follows. For a unit $u\in\mathcal{O}_X(X)^\times$, pullback along the morphism $\mathbb{G}_{\op{m}}\times X\to \mathbb{G}_{\op{m}}\times X$ given by multiplication with $u$ induces an action of $\mathcal{O}_X(X)^\times$ on $\mathbf{A}(\mathbb{G}_{\op{m}}\times X)$. By means of the short exact sequence
\[
0\to \mathbf{A}(X)\xrightarrow{\op{pr}_1^\ast} \mathbf{A}(\mathbb{G}_{\op{m}}\times X)\to \mathbf{A}_{-1}(X)\to 0,
\]
this induces a $\mathbb{G}_{\op{m}}$-action on $\mathbf{A}_{-1}(X)$, noting that the map $\op{pr}_1^\ast$ is $\mathbb{G}_{\op{m}}$-equivariant if we equip $\mathbf{A}(X)$ with the trivial action. This action of the unit group extends to an action of the Grothendieck--Witt ring along the morphism 
\[
\mathbb{G}_{\op{m}}\to\mathbf{K}^{\op{MW}}_0\colon u\mapsto\langle u\rangle=\eta[u]+1
\]
by \cite[Lemma 3.49]{MField}; it can be defined on higher contractions $\mathbf{A}_{-n}$, $n\geq 2$, and is then independent of the choice of which factor of $\mathbb{G}_{\op{m}}^n$ is acting, cf. the discussion after Lemma 3.49 in \cite{MField}.

The unit action above can be used to twist the contractions $\mathbf{A}_{-n}$, $n\geq 1$, of a strictly $\mathbb{A}^1$-invariant sheaf of abelian groups, cf. \cite[p. 118]{MField}. For a field $E$ and an $E$-vector space $V$ of dimension $1$, we denote by $V^\times$ the set $V\setminus\{0\}$ equipped with the induced scalar multiplication of $E^\times$, and define for $n\geq 1$
\[
\mathbf{A}_{-n}(E;V):=\mathbf{A}_{-n}(E)\otimes_{\mathbb{Z}[E^\times]} \mathbb{Z}[V^\times]. 
\]
This is mostly used in the case where $X$ is an essentially smooth $k$-scheme, $E=\kappa(x)$ is the residue field of a point $x\in X^{(n)}$ of codimension $n$ and $\Lambda^X_x:=\wedge^n_{\kappa(x)}(\mathfrak{m}_x/\mathfrak{m}_x^2)^\vee$ the stalk of the conormal sheaf of $x$ in $X$.

Next, residue morphisms for a strictly $\mathbb{A}^1$-invariant sheaf $\mathbf{A}$ are defined as follows, cf. \cite[Corollary 2.35, Lemma 5.10]{MField}. For a field $F$ with discrete valuation $v$ and residue field $E$, we have $\Lambda^{\mathcal{O}_v}_E=\mathfrak{m}_v/\mathfrak{m}_v^2$. Then we choose a uniformizer $\pi$ for $v$. This choice determines a Nisnevich distinguished square 
\[
\xymatrix{
\op{Spec} F\ar[r] \ar[d] & \op{Spec} \mathcal{O}_v \ar[d] \\
\mathbb{G}_{\op{m},E} \ar[r] & \mathbb{A}^1_E
}
\]
by means of which we can identify $\mathbf{A}_{-1}(E)\cong \mathbf{A}(F)/\mathbf{A}(\mathcal{O}_v)$. For a choice of uniformizer $\pi$, the residue map 
\[
\partial^F_E\colon\mathbf{A}(F)\to \mathbf{A}_{-1}(E;\Lambda^{\mathcal{O}_v}_E)=\mathbf{A}_{-1}(E) \otimes_{\mathbb{Z}[E^\times]} \mathbb{Z}[(\Lambda^{\mathcal{O}_v}_E)^\times]
\]
now maps an element $s\in\mathbf{A}(F)$ to $\op{res}(s)\otimes \pi$ where $\op{res}(s)$ is the image of $s$ under the composition $\mathbf{A}(F)\to \mathbf{A}(F)/\mathbf{A}(\mathcal{O}_v)\to \mathbf{A}_{-1}(E)$ of the natural projection with the above isomorphism (which depended on the choice of uniformizer). It turns out, cf. \cite[Lemma 5.10]{MField}, that this definition of residue morphism is independent of the choice of uniformizer. 

If $\mathbf{A}$ is a strictly $\mathbb{A}^1$-invariant sheaf on $\op{Sm}/k$ and $X$ is an essentially smooth $k$-scheme, then there are Gersten-type complexes $\op{C}^\bullet(X;\mathbf{A})$ (called Rost--Schmid complexes in \cite{MField} and \cite{AsokFasel}) which provide flasque resolutions of $\mathbf{A}$ for the Zariski and Nisnevich topology, cf. \cite[Corollary 5.43]{MField}. These complexes have the form 
\[
\cdots \to \bigoplus_{y\in X^{(i)}}\mathbf{A}_{-i}(\kappa(y);\Lambda^X_y)\xrightarrow{\bigoplus\partial^y_z} \bigoplus_{z\in X^{(i+1)}}\mathbf{A}_{-i-1}(\kappa(z);\Lambda^X_z)\to \cdots
\]
The boundary map $\partial^y_z$ is only nontrivial if $z\in\overline{y}$; in this case, it is described as follows, cf. \cite[Section 5.1]{MField}: replace $X$ by the localization at $z$, and let $\tilde{Y}\to Y$ be the normalization of the curve $Y=\overline{y}\hookrightarrow X$. For every point $\tilde{z}_i$ on $\tilde{Y}$ lying over $z\in Y$, we have a residue morphism $\mathbf{A}_{-i}(\kappa(y))\to \bigoplus_{\tilde{z}_i}\mathbf{A}_{-i-1}(\kappa(\tilde{z}_i);\Lambda^{\tilde{Y}}_{\tilde{z}_i})$ as discussed above. We can twist these by the relative canonical bundle $\omega_{\tilde{Y}/X}$ to obtain the twisted residue morphisms 
\[
\mathbf{A}_{-i}(\kappa(y);\Lambda^X_y)\to \mathbf{A}_{-i-1}(\kappa(\tilde{z}_i);\omega_{\tilde{z}_i/X})
\]
noting that $\Lambda^{\tilde{Y}}_{\tilde{z}_i}\otimes \omega_{\tilde{Y}/X}=\omega_{\tilde{z}_i/X}$. The boundary map $\partial^y_z\colon \mathbf{A}_{-i}(\kappa(y);\Lambda^X_y)\to \mathbf{A}_{-i-1}(\kappa(z);\Lambda^X_z)$ is then obtained as the composition of the above twisted residue morphisms with the absolute transfer $\mathbf{A}_{-i-1}(\kappa(\tilde{z}_i);\omega_{\tilde{z}_i/X})\to \mathbf{A}_{-i-1}(\kappa(z);\Lambda^X_z)$. The absolute transfer maps (and the relevant geometric and cohomological transfers) are discussed in Sections 4 and 5.1 of \cite{MField}; we don't recall all the details here since for our computations, the absolute transfers will play no role.\footnote{For most of the computations, the boundary maps in the above complex will not matter; several vanishing results will only be proved by making statements about the structure of the reduction of certain strictly $\mathbb{A}^1$-invariant sheaves. However, there is one cohomology operation which we have to discuss in detail, and this requires a detailed tracing through the construction of the boundary map for the Gersten-type complexes.} If $\mathbf{A}$ is already a contraction of some other strictly $\mathbb{A}^1$-invariant sheaf and $X$ is a smooth scheme with a line bundle $\mathscr{L}$, then the above complexes can be further twisted by the line bundle, but we won't need this for our computations.


One of the noteworthy consequences of the existence of the above complexes computing Nisnevich cohomology is that Zariski and Nisnevich cohomology agree for all the sheaves we will consider in this paper. This fact might be used tacitly in some argument. There may also be occasionally missing indices, but all cohomology groups considered in this paper are Nisnevich cohomology groups.

\subsection{Mayer--Vietoris sequences}
\label{sec:mv}

Finally, we need to discuss Mayer--Vietoris sequences for the Nisnevich cohomology of strictly $\mathbb{A}^1$-invariant sheaves which will be needed for some computations in the paper. Let $X=U\cup V$ be a Zariski covering of smooth schemes over a perfect field $k$, and let $\mathbf{A}$ be a strictly $\mathbb{A}^1$-invariant Nisnevich sheaf of abelian groups on $\op{Sm}/k$. In this situation, it follows formally that there is an associated long exact Mayer--Vietoris sequence in Nisnevich cohomology
\[
\to\op{H}^i_{\op{Nis}}(X,\mathbf{A})\to \op{H}^i_{\op{Nis}}(U,\mathbf{A})\oplus \op{H}^i_{\op{Nis}}(V,\mathbf{A})\to \op{H}^i_{\op{Nis}}(U\cap V,\mathbf{A})\stackrel{\partial}{\longrightarrow} \op{H}^{i+1}_{\op{Nis}}(X,\mathbf{A})\to
\]
This sequence is functorial in $\mathbf{A}$.

One can also use the Gersten complexes to provide a specific model of the boundary map in the Mayer--Vietoris sequence as follows. There is an exact sequence 
\[
0\to \op{C}^\bullet(X,\mathbf{A})\to\op{C}^\bullet(U,\mathbf{A}) \oplus \op{C}^\bullet(V,\mathbf{A})\to \op{C}^\bullet(U\cap V,\mathbf{A})\to 0
\]
of Gersten complexes associated to a Zariski covering $X=U\cup V$
which gives rise to the long exact Mayer--Vietoris sequence. The boundary morphism can now be described as follows. A cycle representative of a class in $\op{H}^i_{\op{Nis}}(U\cap V,\mathbf{A})$ is given by a finite sum, indexed over codimension $i$ points $x$ of $U\cap V$, of elements of $\mathbf{A}_{-i}(k(x))$. This formal sum can be viewed as a formal sum in $\op{C}^\bullet(U;\mathbf{A})$ (which is the summand of the middle complex where the restriction map is $\op{id}$, not $-\op{id}$). This may no longer be a cycle, but  we can apply the boundary map of the Gersten complex (which is given in terms of residue maps) to this lifted chain. The result will have trivial image in $\op{C}^\bullet(U\cap V;\mathbf{A})$ and therefore the resulting formal sum can be viewed as an element of $\op{C}^\bullet(X;\mathbf{A})$. This will be a cycle representing the image of the boundary map $\op{H}^i_{\op{Nis}}(U\cap V,\mathbf{A})\to\op{H}^{i+1}_{\op{Nis}}(X,\mathbf{A})$.

\subsection{Quadrics, unimodular rows and cohomology}
\label{sec:quadrics}

We denote by $\op{Q}_d$ the $d$-dimensional smooth affine split quadric, cf.~\cite{AsokDoranFasel} for the defining equations over $\mathbb{Z}$. 

For the odd-dimensional quadrics $\op{Q}_{2n-1}$, defined by the equation $\sum_{i=1}^n X_iY_i=1$, it is classical that the projection onto the $X_i$-coordinates provides an $\mathbb{A}^1$-equivalence $\op{Q}_{2n-1}\to\mathbb{A}^n\setminus\{0\}$. Moreover, the odd-dimensional quadric $\op{Q}_{2n-1}$ is $\mathbb{A}^1$-local in the sense that the set $[X,\op{Q}_{2n-1}]_{\mathbb{A}^1}$ of $\mathbb{A}^1$-homotopy classes of maps into $\op{Q}_{2n-1}$ can be identified as the quotient of the set of scheme morphisms $X\to\op{Q}_{2n-1}$ modulo naive $\mathbb{A}^1$-homotopies $X\times \mathbb{A}^1\to \op{Q}_{2n-1}$, for every smooth affine scheme $X$; this is explained e.g. in \cite[\S 4]{AsokFaselSpheres}. 

For a commutative unital $k$-algebra $R$, a tuple $(a_1,\dots,a_n)$ of elements of $R$ is called \emph{unimodular row} if there exists a tuple $(b_1,\dots,b_n)$ of elements of $R$ such that $\sum a_ib_i=1$. The scheme $\mathbb{A}^n\setminus\{0\}$ classifies unimodular rows in the sense that for $X=\op{Spec} R$ a smooth affine $k$-scheme the set $[\op{Spec} R,\mathbb{A}^n\setminus\{0\}]_{\mathbb{A}^1}$ is in natural bijection with the orbit set of the natural action of the elementary group $\op{E}_n(R)$ on the set $\op{Um}_n(R)$ of unimodular rows over $R$. For the odd-dimensional quadric $\op{Q}_{2n-1}$, the scheme morphisms $X\to\op{Q}_{2n-1}$ are in bijection with pairs of a unimodular row $(a_1,\dots,a_n)$ and a choice of tuple $(b_1,\dots,b_n)$ with $\sum a_ib_i=1$. Any unimodular row $X\to\mathbb{A}^n\setminus\{0\}$ can be lifted to a morphism $X\to\op{Q}_{2n-1}$ and any two such lifts are equivalent up to $\mathbb{A}^1$-homotopy. For a further discussion of these issues, cf. again \cite[\S 4]{AsokFaselSpheres}. 

The relation between unimodular rows and $\mathbb{A}^n\setminus\{0\}$ resp. $\op{Q}_{2n-1}$ is relevant because of the following. A unimodular row $(a_1,\dots,a_n)$ of length $n$ over the ring $R$ gives rise to a stably free projective module of rank $n-1$, given as kernel of the map $R^n\to R\colon(b_1,\dots,b_n)\mapsto \sum a_ib_i$ defined by the unimodular row. On the level of $\mathbb{A}^1$-homotopy theory, this is reflected by the $\mathbb{A}^1$-fiber sequence
\[
\mathbb{A}^n\setminus\{0\}\to {\op{B}}\op{SL}_{n-1}\to {\op{B}}\op{SL}_n.
\]
Composing the morphism $X\to\mathbb{A}^n\setminus\{0\}$ corresponding to the unimodular row with the map $\mathbb{A}^n\setminus\{0\}\to{\op{B}}\op{SL}_{n-1}$ yields a rank $n-1$ vector bundle which becomes trivial after adding a trivial line bundle. This is what is relevant for the cohomological analysis of Mohan Kumar's stable free vector bundles: we will establish the non-triviality of the vector bundles by showing that the associated morphism $X\to \mathbb{A}^n\setminus\{0\}\to{\op{B}}\op{SL}_{n-1}$ is not null-homotopic. The obstruction-theoretic analysis reduces this to show that some cohomology classes of $X$ with coefficients in $\bm{\pi}^{\mathbb{A}^1}_i({\op{B}}\op{SL}_{n-1})$ are non-trivial. 

Since we will also need to compute the Nisnevich cohomology of the quadrics $\op{Q}_d$ with coefficients in strictly $\mathbb{A}^1$-invariant sheaves, we shortly recall the relevant formulas, cf. \cite{AsokDoranFasel}:
\[
\tilde{\op{H}}^i(\op{Q}_{2d},\mathbf{A})\cong\left\{\begin{array}{ll}
\mathbf{A}_{-d}(k) & i=d\\
0 & \textrm{else}
\end{array}\right.,\qquad
\tilde{\op{H}}^i(\op{Q}_{2d-1},\mathbf{A})\cong\left\{\begin{array}{ll} \mathbf{A}_{-d}(k) & i=d-1\\
0 & \textrm{else}
\end{array}\right.
\]
Here $k$ is the base field and $\mathbf{A}_{-d}$ denotes the $d$-fold contraction. 

\section{Cohomological analysis of Mohan Kumar's construction}
\label{sec:mkgeometry}

In this section, we recall the geometric constructions of \cite{mohan:kumar} and explain how they give rise to varieties with interesting cohomology classes. In fact, we will explain how Mohan Kumar's construction provides varieties $Y\cap Z$ where the following composition is a surjection
\[
[Y\cap Z,\mathbb{A}^{p+1}\setminus\{0\}]_{\mathbb{A}^1}\to \op{H}^p(Y\cap Z,\mathbf{K}^{\op{MW}}_{p+1})\to \op{H}^p (Y\cap Z,\mathbf{K}^{\op{M}}_{p+1}/p)\to \mathbb{Z}/p\mathbb{Z}.
\]

\subsection{Geometric setup} 
Fix a prime $p$ and a field $k$. The first geometric construction produces a smooth affine variety with non-trivial torsion in the top Chow group. For this, let $f(T)$ be a polynomial of degree $p$ over $k$ such that $f(0)=a\in k^\times$. Then there are recursively defined polynomials 
\begin{eqnarray*}
F_1(X_0,X_1)&=&X_1^pf\left(\frac{X_0}{X_1}\right), \textrm{ and}\\
F_{i+1}(X_0,\dots,X_{i+1})&=& F_1\left(F_i(X_0,\dots,X_i),a^{\frac{p^i-1}{p-1}}X_{i+1}^{p^i}\right).
\end{eqnarray*} 
If $f(T^{p^{m-1}})$ is irreducible then, according to \cite[Claim 1]{mohan:kumar}, $F_n$ is irreducible for $n\leq m$. In this case, \cite[Claim 2]{mohan:kumar} states that $X=\mathbb{P}^n\setminus \op{V}(F_n)$ is a smooth affine variety over $k$ with $\op{CH}^n(X)\cong\mathbb{Z}/p\mathbb{Z}$, generated by the class of a $k$-rational point of $X$. 

The second part of the geometric construction produces a Zariski covering  of $X$ by two affine subvarieties with trivial top Chow groups. The first subvariety is 
\[
Y=\left(\mathbb{P}^n\setminus\op{V}(F_{n-1})\right)\cap X,
\]
where we view $F_{n-1}$ in the obvious way as a polynomial in the variables $X_0,\dots,X_n$. Since $X\setminus Y$ contains the $k$-rational point $x=[0:0:\cdots:0:1]$, we have $\op{CH}^n(Y)=0$ by the localization sequence for Chow groups. The second subvariety is 
\[
Z=\left(\mathbb{P}^n\setminus\op{V}(G)\right)\cap X
\]
with the polynomial 
\[
G(X_0,\dots,X_n)=F_{n-1}(X_0,\dots,X_{n-1})-a^{\frac{p^{n-1}-1}{p-1}}X_n^{p^{n-1}}.
\]
Here the variety $\op{V}(G)$ contains the $k$-rational point $y=[0:0:\cdots:0:1:1]$, and again the localization sequence for Chow groups implies $\op{CH}^n(Z)=0$. We have a Zariski covering $X=Y\cup Z$ because $F_n\in\langle F_{n-1},G\rangle$.

Finally, the relevant variety is now the intersection $Y\cap Z$. 

\subsection{Non-trivial cohomology classes}

We first note that the variety $Y\cap Z$ constructed by Mohan Kumar supports a non-trivial cohomology class with coefficients in Milnor--Witt K-theory. This non-trivial class exists because the class in $\op{CH}^n(X)$ locally trivializes in the covering $X=Y\cup Z$. 

\begin{proposition}
\label{prop:mkclass}
Let $k$ be a field of $2$-cohomological dimension $\leq 1$. Let $p$ be a prime and assume that there exists a degree $p$ polynomial $f(T)$ over $k$ such that $f(T^{p^{p}})$ is irreducible. Consider the situation $X=Y\cup Z$ outlined above. Then there is a surjection
\[
\op{H}^p_{\op{Nis}}(Y\cap Z,\mathbf{K}^{\op{MW}}_{p+1})\twoheadrightarrow\op{CH}^{p+1}(X)\cong\mathbb{Z}/p\mathbb{Z}.
\]
\end{proposition}

\begin{proof}
Use the Mayer--Vietoris sequence for the cohomology of $\mathbf{K}^{\op{M}}_{p+1}$  associated to the Zariski covering $X=Y\cup Z$, whose relevant portion is the following
\[
\op{H}^p_{\op{Nis}}(Y\cap Z,\mathbf{K}^{\op{M}}_{p+1})\to\op{CH}^{p+1}(X)\to \op{CH}^{p+1}(Y)\oplus\op{CH}^{p+1}(Z).
\]
By construction the last group of the sequence is trivial, showing that the boundary map in the Mayer--Vietoris sequence is a surjection. Now we consider the  exact sequence of strictly $\mathbb{A}^1$-invariant sheaves $0\to\mathbf{I}^{p+2}\to\mathbf{K}^{\op{MW}}_{p+1}\to\mathbf{K}^{\op{M}}_{p+1}\to 0$. The induced morphism $\op{H}^p_{\op{Nis}}(Y\cap Z,\mathbf{K}^{\op{MW}}_{p+1})\to \op{H}^p_{\op{Nis}}(Y\cap Z,\mathbf{K}^{\op{M}}_{p+1})$ is surjective if we can show  $\op{H}^{p+1}_{\op{Nis}}(Y\cap Z,\mathbf{I}^{p+2})=0$. But that follows from \cite[Proposition 5.2]{AsokFasel} because by assumption the $2$-cohomological dimension of the base field $k$ is $\leq 1$.
\end{proof}

\begin{corollary}
Let $k$ be a field of $2$-cohomological dimension $\leq 1$. Let $p$ be a prime and assume that there exists a degree $p$ polynomial $f(T)$ over $k$ such that $f(T^{p^{p}})$ is irreducible. Consider the situation $X=Y\cup Z$ outlined above. Then there is a surjection
\[
[Y\cap Z,\op{Q}_{2p+1}]_{\mathbb{A}^1}\to \op{H}^p_{\op{Nis}}(Y\cap Z,\mathbf{K}^{\op{MW}}_{p+1})\to\op{CH}^{p+1}(X)\cong\mathbb{Z}/p\mathbb{Z}.
\]
\end{corollary}

\begin{proof}
By Proposition~\ref{prop:mkclass} it suffices to show that the first map is a surjection. This follows from \cite[Proposition 1.1.10 (1)]{AsokFaselCohomotopy} since $\op{Q}_{2p+1}$ is $(p-1)$-$\mathbb{A}^1$-connected and $Y\cap Z$ has Krull dimension $p+1$.
\end{proof}

\begin{remark}
It would be very interesting to have more generally a construction of smooth affine varieties with non-trivial classes in $\op{H}^n_{\op{Nis}}(X,\mathbf{K}^{\op{M}}_{n+r}/m)$ for $r\geq 2$. Possibly this could be done by setting up a Mayer--Vietoris spectral sequence for coverings $X=U_1\cup\cdots\cup U_{r+1}$ in Milnor K-cohomology and then use that to produce such varieties as intersections $U_1\cap\cdots\cap U_{r+1}$. The next interesting case would be $r=2$. For this, one would want a variety $X$ with a covering $X=U_1\cup U_2\cup U_3$, a non-trivial class in $\op{CH}^{n+2}(X)/m$ whose restriction to $U_i$ vanishes \emph{and} such that the induced classes in $\op{H}^{n+1}_{\op{Nis}}(U_i\cap U_j,\mathbf{K}^{\op{M}}_{n+2}/m)$ are also trivial at $E_\infty$ (either by not being a cycle in $E_1$ or by lifting to $U_i$ and thus be killed by the $\op{d}^1$-differential). The Mayer--Vietoris spectral sequence would then produce a non-trivial class in $\op{H}^n_{\op{Nis}}(U_1\cap U_2\cap U_3,\mathbf{K}^{\op{M}}_{n+2}/m)$. However, at this point I don't know how to guarantee the latter condition on the vanishing of $\op{H}^{n+1}_{\op{Nis}}(\mathbf{K}^{\op{M}}_{n+2}/m)$. As pointed out by one of the referees, the base field would have to be of higher cohomological dimension, such as a function field in $r$ variables, to have a chance for the resulting cohomology classes to be nontrivial. 
\end{remark}

\subsection{Explicit description of a class}
\label{sec:explicit}

We now want to write out an explicit description of a Milnor K-cohomology class whose boundary can be detected on the Chow group. Note that the smooth affine variety $X$ is defined as the complement of a hypersurface such that each point on the hypersurface has degree divisible by $p$ (which is exactly the reason for the $p$-torsion in $\op{CH}^{p+1}(X)$). Now the complements of $Y$ and $Z$ in $X$ are given by hypersurfaces which contain rational points (which is the reason why the top Chow groups of $Y$ and $Z$ are trivial). Note however that the hypersurface complements of $Y$ and $Z$ only intersect in the complement of $X$, so they don't have rational points in common. On $Y\cap Z$ there are hence two reasons for triviality of $\op{CH}^{p+1}(Y\cap Z)$, namely  a rational point in the complement of $Y$ or a rational point in the complement of $Z$. The lift in the Mayer--Vietoris sequence is then given by a ``homotopy between these two trivializations'': take a line in $\mathbb{P}^{p+1}$ connecting a rational point in $X\setminus Y$ and a rational point in $X\setminus Z$. On this line there is a rational function having divisor exactly the difference of these rational points. The class of this rational function in the $p$-residues of the function field of the line is a cycle on $Y\cap Z$, hence represents a class in $\op{H}^p_{\op{Nis}}(Y\cap Z,\mathbf{K}^{\op{M}}_{p+1}/p)$. Note that this is exactly the geometric situation in \cite[Claim 3]{mohan:kumar}, and we will see in \ref{sec:comparison} that this is the lifting class of Mohan Kumar's stably free module. 

Now we want to show that the  class we described in $\op{H}^p_{\op{Nis}}(Y\cap Z,\mathbf{K}^{\op{M}}_{p+1}/p)$ is actually non-trivial. This non-triviality is detected using the Mayer--Vietoris sequence associated to the covering $X=Y\cup Z$, cf. \ref{sec:mv}. More specifically, we want to show that the above cycle has non-trivial image under the boundary map
\[
\op{H}^p_{\op{Nis}}(Y\cap Z,\mathbf{K}^{\op{M}}_{p+1}/p)\to \op{CH}^{p+1}(X)/p.
\]
To compute the image of the class, recall the description of the boundary map from \ref{sec:mv}. The cycle description of the class above was that we take the rational function with divisor $[y]-[x]$ as an element in the mod $p$ residues of the function field of the line connecting the points $x$ and $y$. To compute the boundary in the Mayer--Vietoris sequence, we first take the Gersten chain on $Y$ given by the very same rational function on the very same line. Now we apply the boundary map in the Gersten complex, which in our situation is given by mapping the rational function on the line $l\cap Y$ to its divisor. The rational function on the line $l$ has a zero at $y\in Y$ and a pole at $x\in X$, hence its divisor on the line $l$ is $[y]-[x]$. However, since $x\not\in Y$, the divisor of the function in the Gersten complex for $Y$ is $[y]$. The final step in the computation of the boundary for the Mayer--Vietoris sequence is to view $[y]$ as a $0$-cycle on the whole variety $X$. Since $y$ is a $k$-rational point, the class $[y]$ is a generator of $\op{CH}^{p+1}(X)$. 

We have shown the following:
\begin{proposition}
\label{prop:explicit}
Denote by $\sigma\in\op{H}^p_{\op{Nis}}(Y\cap Z,\mathbf{K}^{\op{M}}_{p+1}/p)$ the  class of a rational function with divisor $[y]-[x]$ supported on the line connecting $x$ and $y$. Then the image of $\sigma$ under the boundary map 
\[
\op{H}^p_{\op{Nis}}(Y\cap Z,\mathbf{K}^{\op{M}}_{p+1}/p)\to \op{CH}^{p+1}(X)/p
\]
is a generator of $\op{CH}^{p+1}(X)/p\cong\mathbb{Z}/p\mathbb{Z}$. 
\end{proposition}

\begin{remark}
Another approach to the construction of a non-trivial class in the group $\op{H}^{p-1}_{\op{Nis}}(W,\mathbf{K}^{\op{M}}_{p+1}/p)$ could now be to use the construction above. If we can provide a covering $Y\cap Z=W_1\cup W_2$ on which the  Milnor K-cohomology class in $\op{H}^p_{\op{Nis}}(Y\cap Z,\mathbf{K}^{\op{M}}_{p+1}/p)$ is trivialized, then the Mayer--Vietoris sequence would produce such a class. The Milnor K-cohomology class can be represented by a line in $\mathbb{P}^n$ connecting the points $x$ and $y$. Actually, using Milnor $\op{K}_2$-classes associated to the coordinate axes in $\mathbb{P}^2$, any line in $\mathbb{P}^{p+1}$ connecting rational points on $X\setminus Y$ and $X\setminus Z$ will represent the same cohomology class. If we now can find two such lines $L_1,L_2$ contained in hypersurfaces $S_1,S_2$ which only meet outside $Y\cap Z$, then the Mayer--Vietoris sequence associated to the covering $Y\cap Z=(Y\cap Z\setminus S_1)\cup (Y\cap Z\setminus S_2)$ would produce the required class.  Unfortunately, I don't know how to construct the appropriate hypersurfaces. 
\end{remark}

\subsection{Comparison with Mohan Kumar's construction}
\label{sec:comparison}

Now we want to compare this to the original construction of stably free modules in \cite{mohan:kumar}. In fact, we will show that the unimodular row defining Mohan Kumar's stably free module maps exactly to a cohomology class as described above.

Recall from \cite{mohan:kumar} that the stably free module $\mathscr{P}$ is given by a unimodular row as follows. The point $y$ is a complete intersection in $Y$, i.e., its maximal ideal is of the form $\mathfrak{m}_y=\langle b_1,\dots,b_n\rangle$ for a regular sequence of functions $b_1,\dots,b_n\in\mathcal{O}_Y(Y)$. Because $y\not\in  Z$, we can now consider $(b_1,\dots,b_n)$ as a unimodular row on $Y\cap Z$. The stably free module $\mathscr{P}$ over $Y\cap Z$ is now the one defined as the kernel of this unimodular row, cf. p.1441 of \cite{mohan:kumar}. Note that, compared to the situation in \ref{sec:explicit}, the regular sequence $(b_2,\dots,b_n)$ defines the intersection of the line connecting $x$ and $y$ with $Y$  in $Y$, cf. \cite[proof of Claim 3]{mohan:kumar}, and $b_1$ can be taken to be a function having a simple zero at $y$ and a simple pole at $x$.

It remains to recall the description of the map $[Y\cap Z,\mathbb{A}^{p+1}\setminus\{0\}]_{\mathbb{A}^1}\to\op{H}^p(Y\cap Z,\mathbf{K}^{\op{MW}}_{p+1})$ sending a map $Y\cap Z\to\mathbb{A}^{p+1}\setminus\{0\}$ corresponding to a unimodular row of length $p+1$ to its associated lifting class. Recall that for each unimodular row of length $p+1$ over the ring $R$ there is a morphism $u\colon\op{Spec} R\to \op{Q}_{2p+1}\simeq_{\mathbb{A}^1}\mathbb{A}^{p+1}\setminus \{0\}$, well-defined up to $\mathbb{A}^1$-homotopy. The first lifting class associated to the morphism $Y\cap Z\to\op{Q}_{2p+1}$ is given by the composition 
\[
Y\cap Z\to\op{Q}_{2p+1}\to \tau_{\leq p}\op{Q}_{2p+1}\cong \op{K}(\mathbf{K}^{\op{MW}}_{p+1},p).
\]
This composition corresponds to a cohomology class in $\op{H}^p_{\op{Nis}}(Y\cap Z,\mathbf{K}^{\op{MW}}_{p+1})$ which can be evaluated using the techniques discussed in the proof of \cite[Theorem 4.1]{fasel:unimodular}: without loss of generality, we can assume that the unimodular row $(b_1,\dots,b_{p+1})$ is such that the sequence $(b_2,\dots,b_{p+1})$ is regular; the first lifting class of the unimodular row in $\op{H}^p_{\op{Nis}}(Y\cap Z,\mathbf{K}^{\op{MW}}_{p+1})$ is then given by the cycle $(b_1,\langle -1,b_1\rangle)$ on the subscheme defined by $(b_2,\dots,b_{p+1})$. Reduction of coefficients to $\mathbf{K}^{\op{M}}_{p+1}/p!$ means that the class is given by the unit $b_1$ on the closed integral subscheme defined by $(b_2,\dots,b_{p+1})$.

We formulate the combination of the above statements which is implicitly contained in \cite[Theorem 4.1]{fasel:unimodular}:

\begin{proposition}
\label{prop:lifting}
Let $k$ be an infinite field, and let $X=\op{Spec}R$ be a smooth affine scheme over $k$. Let $(b_1,\dots,b_{p+1})$ be a unimodular row over $R$ such that the sequence $(b_2,\dots,b_{p+1})$ is regular and denote by $P$ the associated stably free module of rank $p$. The first lifting class associated to $P$ in $\op{H}^p_{\op{Nis}}(X,\mathbf{K}^{\op{M}}_{p+1}/p!)$ is given by the cycle whose underlying codimension $p$ scheme is $R/(b_2,\dots,b_{p+1})$ and the associated rational function on it is $b_1$.
\end{proposition}

Applying this statement to the specific case of Mohan Kumar's stably free modules, the lifting class is given by a rational function with divisor $[y]-[x]$ viewed as $p$-residue in the function field of the line connecting $x$ and $y$, cf. \cite[Claim 3]{mohan:kumar}. Note that this is exactly the class discussed in \ref{sec:explicit}. In particular, Proposition~\ref{prop:explicit} now implies that the lifting class of Mohan Kumar's stably free module is a non-trivial class in $\op{H}^p_{\op{Nis}}(Y\cap Z,\mathbf{K}^{\op{M}}_{p+1}/p!)$, mapping to a generator of $\op{CH}^{p+1}(X)/p!\cong\mathbb{Z}/p\mathbb{Z}$ under the boundary map of the Mayer--Vietoris sequence. 

It should be mentioned that actually writing down an explicit unimodular row for Mohan Kumar's construction is a difficult task. The existence of the unimodular row only requires knowing that $y$ is a complete intersection point, while writing down an explicit unimodular row requires finding an explicit regular sequence generating the ideal defining the point $y$.

\section{Stably free modules: Mohan Kumar's examples at odd primes}

In the previous section, we discussed a cohomological reinterpretation of Mohan Kumar's constructions from \cite{mohan:kumar}. This provided, in particular, a smooth affine variety $Y\cap Z$ with a morphism $Y\cap Z\to\mathbb{A}^{p+1}\setminus\{0\}$ which is detected in Milnor K-cohomology. This morphism can be composed with the natural map $\mathbb{A}^{p+1}\setminus\{0\}\to{\op{B}}\op{SL}_p$ which is the inclusion of the $\mathbb{A}^1$-homotopy fiber of the stabilization map ${\op{B}}\op{SL}_p\to {\op{B}}\op{SL}_{p+1}$. This produces a rank $p$ vector bundle which becomes trivial after adding a free rank one summand; it is an $\mathbb{A}^1$-topological reformulation of the fact that the kernel of a unimodular row of length $p+1$ is a projective module of rank $p$ which becomes trivial after adding a free rank one module. The morphism $\mathbb{A}^{p+1}\setminus\{0\}\to{\op{B}}\op{SL}_p$ induces a morphism $\op{H}^p(Y\cap Z,\mathbf{K}^{\op{MW}}_{p+1})\to \op{H}^p(Y\cap Z,\bm{\pi}^{\mathbb{A}^1}_p{\op{B}}\op{SL}_p)$.  The main point of the present section will now be to show that the composition
\[
[Y\cap Z,\mathbb{A}^{p+1}\setminus\{0\}]_{\mathbb{A}^1}\to \op{H}^p(Y\cap Z,\mathbf{K}^{\op{MW}}_{p+1}) \to \op{H}^p(Y\cap Z,\bm{\pi}^{\mathbb{A}^1}_p{\op{B}}\op{SL}_p)
\]
sends the unimodular row given by Mohan Kumar to a non-zero element, thus providing an $\mathbb{A}^1$-topological proof that the stably free module is non-trivial. For this, we will need to recall some information on the $\mathbb{A}^1$-homotopy sheaf $\bm{\pi}^{\mathbb{A}^1}_p({\op{B}}\op{SL}_p)$. 

For the present section, $p$ will be an odd prime, the case $p=2$ will be discussed in the next section. This case distinction is due to a structural difference of the relevant $\mathbb{A}^1$-homotopy sheaves; Mohan Kumar's constructions in \cite{mohan:kumar} work for even and odd primes.

\subsection{$\mathbb{A}^1$-homotopy groups and a cohomology operation}
\label{sec:a1homotopy}

The first thing to note is that ${\op{B}}\op{SL}_n$ is simply connected for all $n$. The stabilization results imply that 
\[
\bm{\pi}^{\mathbb{A}^1}_i({\op{B}}\op{SL}_n)\cong \mathbf{K}^{\op{Q}}_i
\]
for $i<n$. The corresponding lifting classes are related to the Chern classes, but the non-uniqueness of lifting classes implies that making this relationship precise is a rather subtle business.

The first unstable $\mathbb{A}^1$-homotopy group of ${\op{B}}\op{SL}_n$ has been computed in \cite[Theorem 3.14]{AsokFaselSpheres}; for odd $n$, it is given by an exact sequence 
\[
0\to\mathbf{S}_{n+1}\to\bm{\pi}^{\mathbb{A}^1}_n({\op{B}}\op{SL}_n)\to \mathbf{K}^{\op{Q}}_n\to 0. 
\]
This exact sequence arises from the long exact homotopy sequence for the stabilization fiber sequence $\mathbb{A}^{n+1}\setminus\{0\}\to{\op{B}}\op{SL}_n\to{\op{B}}\op{SL}_{n+1}$; consequently, the sheaf $\mathbf{S}_{n+1}$ is the cokernel of the boundary map $\bm{\pi}^{\mathbb{A}^1}_{n+1}({\op{B}}\op{SL}_{n+1})\to \bm{\pi}^{\mathbb{A}^1}_{n}(\mathbb{A}^{n+1}\setminus\{0\})\cong\mathbf{K}^{\op{MW}}_{n+1}$. By \cite[Corollary 3.11]{AsokFaselSpheres}, the canonical epimorphism $\mathbf{K}^{\op{MW}}_{n+1}\to\mathbf{S}_{n+1}$ factors through a canonical epimorphism $\mathbf{K}_{n+1}^{\op{M}}/n!\to \mathbf{S}_{n+1}$ which becomes an isomorphism after $n-1$-fold contraction. 

The main point of the present section is the analysis of a cohomology operation associated to the $\mathbb{A}^1$-homotopy sheaf $\bm{\pi}^{\mathbb{A}^1}_n({\op{B}}\op{SL}_n)$. The short exact sequence of strictly $\mathbb{A}^1$-invariant sheaves above induces a long exact sequence in Nisnevich cohomology whose boundary map has the form 
\[
\op{CH}^n(X)\cong\op{H}^n(X,\mathbf{K}^{\op{Q}}_n)\to \op{H}^{n+1}(X,\mathbf{S}_{n+1})\cong \op{CH}^{n+1}(X)/n!
\]
For later arguments, we want to show that this boundary map is almost trivial. 

\begin{proposition}
\label{prop:cohop}
Let $k$ be a field, let $X$ be a smooth $k$-variety and let $n$ be an odd integer. Then the composition of the boundary map $\op{CH}^n(X)\to \op{CH}^{n+1}(X)/n!$ induced from the exact sequence 
\[
0\to \mathbf{S}_{n+1}\to \bm{\pi}^{\mathbb{A}^1}_n({\op{B}}\op{SL}_n)\to \mathbf{K}^{\op{Q}}_n\to 0
\]
with the natural reduction map $\op{CH}^{n+1}(X)/n!\to\op{CH}^{n+1}(X)/n$ is trivial. 
\end{proposition} 

The proof requires tracing through the definition of the boundary map together with some knowledge of the extension class of the exact sequence describing $\bm{\pi}_n^{\mathbb{A}^1}({\op{B}}\op{SL}_n)$ and the unit action on contractions of this sheaf. That is, we need to prove a few lemmas before we can get to the proof of Proposition~\ref{prop:cohop}. We first discuss the explicit realization of $\left(\bm{\pi}^{\mathbb{A}^1}_n{\op{B}}\op{SL}_n\right)_{-n}$ in terms of vector bundles. 

\begin{lemma}
\label{lem:one}
Let $F$ be a field. There is a natural identification 
\[
\op{H}^n(\op{Q}_{2n,F},\bm{\pi}^{\mathbb{A}^1}_n({\op{B}}\op{SL}_n))\cong \left(\bm{\pi}^{\mathbb{A}^1}_n({\op{B}}\op{SL}_n)\right)_{-n}(F)\cong [\op{Q}_{2n,F},{\op{B}}\op{SL}_n]_{\mathbb{A}^1},
\]
and the right-hand side can be further identified with the set of isomorphism classes of vector bundles of rank $n$ over $\op{Q}_{2n,F}$. Moreover, the $n$-fold contraction  
\[
\left(\bm{\pi}^{\mathbb{A}^1}_n({\op{B}}\op{SL}_n)\right)_{-n}\to\mathbb{Z}
\]
of the projection map in the above exact sequence maps a rank $n$ vector bundle on $\op{Q}_{2n,F}$ to its class in $\tilde{\op{K}}_0(\op{Q}_{2n})\cong\mathbb{Z}$. 
\end{lemma}

\begin{proof}
Concerning the natural identifications in the first statement, the first isomorphism follows from the computation of cohomology of quadrics in Section~\ref{sec:quadrics} and the second isomorphism follows from obstruction theory, cf. Section~\ref{sec:representable}. The identification with vector bundles is the representability theorem~\ref{thm:representability}. The final statement concerning the projection follows since the projection $\bm{\pi}^{\mathbb{A}^1}_n{\op{B}}\op{SL}_n\to \mathbf{K}^{\op{Q}}_n$ is induced from the stabilization morphism ${\op{B}}\op{SL}_n\to {\op{B}}\op{SL}_\infty$. See also the discussion of stable vector bundles on $\op{Q}_{2n}$ in \cite{AsokDoranFasel}. 
\end{proof}

\begin{remark}
\label{rem:vblift}
Note that this identification is natural in the underlying field: for a field extension $E/F$ the restriction map on the sections of $\bm{\pi}^{\mathbb{A}^1}_n({\op{B}}\op{SL}_n)_{-n}$ corresponds to pullback of vector bundles along the base-change morphism $\op{Q}_{2n,E}\to \op{Q}_{2n,F}$. In particular, if we choose a preimage of $1\in\mathbb{Z}$ in $\bm{\pi}^{\mathbb{A}^1}_n({\op{B}}\op{SL}_n)_{-n}(k)$ over the base field $k$, then the pullback of the corresponding vector bundle to an extension field $F/k$ will map to $1$ under $\left(\bm{\pi}^{\mathbb{A}^1}_n({\op{B}}\op{SL}_n)\right)_{-n}(F)\to\mathbb{Z}$. 
Explicitly, a vector bundle $\mathcal{V}\to\op{Q}_{2n,F}$ mapping to a generator can be described in terms of Suslin matrices, cf. \cite{AsokDoranFasel}; this is one particular choice of lift of $1$ which can be defined over the base field $k$. 
\end{remark}

\begin{lemma}
\label{lem:two}
Let $F$ be a field with discrete valuation $v$ and residue field $E$. For an element $\mathcal{V}\in \left(\bm{\pi}^{\mathbb{A}^1}_n({\op{B}}\op{SL}_n)\right)_{-n}(F)$ and a unit $u\in F^\times$, the element 
\[
\partial^F_E(\langle u\rangle\mathcal{V})-\partial^F_E(\mathcal{V})\in \mathbb{Z}/n!\mathbb{Z}
\]
is 2-torsion.
\end{lemma}

\begin{proof}
For the action of the unit group $F^\times$ on $\left(\bm{\pi}^{\mathbb{A}^1}_n({\op{B}}\op{SL}_n)\right)_{-n}(F)$, we have $\langle u\rangle \mathcal{V}=\mathcal{V}+\eta[u]\mathcal{V}$. Since the residue map $\partial^F_E$ is additive, we have $\partial^F_E(\langle u\rangle \mathcal{V})=\partial^F_E(\mathcal{V})+\partial^F_E(\eta[u]\mathcal{V})$ and it suffices to show that $\partial^F_E(\eta[u]\mathcal{V})=\eta\partial^F_E([u]\mathcal{V})$ is 2-torsion. To show that the map
\[
\eta\colon \left(\bm{\pi}^{\mathbb{A}^1}_n({\op{B}}\op{SL}_n)\right)_{-n}(F)\to \left(\bm{\pi}^{\mathbb{A}^1}_n({\op{B}}\op{SL}_n)\right)_{-n-1}(F)\cong \mathbb{Z}/n!\mathbb{Z}
\]
lands in the 2-torsion, we note that $\mathbf{K}^{\op{M}}_{n+1}$ is orientable, and so $\eta$ acts as $0$ on $\mathbf{K}^{\op{M}}_1/n!$. In particular, the map $\eta$ above factors through the stabilization projection $\left(\bm{\pi}^{\mathbb{A}^1}_n({\op{B}}\op{SL}_n)\right)_{-n}(F)\to\mathbb{Z}$ and it suffices to show that $\eta\colon\op{K}^{\op{Q}}_0(F)\cong\mathbb{Z}\to\mathbb{Z}/n!\mathbb{Z}\cong \op{K}^{\op{M}}_0(F)/n!$ is 2-torsion. To check this, we consider the following explicit model of $\eta$. We view $\left(\bm{\pi}^{\mathbb{A}^1}_n{\op{B}}\op{SL}_n\right)_{-i}(F)$ for $i=n$ and $n+1$ as $\mathbb{A}^1$-homotopy classes of maps $\op{Q}_{2n}\to{\op{B}}\op{SL}_n$ and $\op{Q}_{2n+1}\to{\op{B}}\op{SL}_n$, respectively, and then multiplication by $\eta$ is explicitly given by composing a morphism $\op{Q}_{2n,F}\to{\op{B}}\op{SL}_n$ with the Hopf map $\eta\colon\op{Q}_{2n+1,F}\to\op{Q}_{2n,F}$. The claim that $\eta\colon\left(\bm{\pi}^{\mathbb{A}^1}_n({\op{B}}\op{SL}_n)\right)_{-n}(F)\to \mathbb{Z}/n!\mathbb{Z}$ is 2-torsion can then be checked in topological or \'etale realization over algebraically closed fields -- it's the classical statement that the Hopf map $\eta\in \pi_{n+1}(\op{S}^n)$ is 2-torsion.
\end{proof}

\begin{lemma}
\label{lem:three}
Let $F$ be a field with discrete valuation $v$. Under the residue map 
\[
\left(\bm{\pi}^{\mathbb{A}^1}_n{\op{B}}\op{SL}_n\right)_{-n}(F)\to \mathbb{Z}/n!\mathbb{Z}
\]
the vector bundle $\mathscr{V}\to \op{Q}_{2n,F}$ from Remark~\ref{rem:vblift} maps to $0$. 
\end{lemma}

\begin{proof}
If $v$ is the discrete valuation on $F$ with residue field $E$, the residue is a map
\[
\left(\bm{\pi}^{\mathbb{A}^1}_n({\op{B}}\op{SL}_n)\right)_{-n}(F)\to \left(\bm{\pi}^{\mathbb{A}^1}_n({\op{B}}\op{SL}_n)\right)_{-n-1}(E;\Lambda^{\mathcal{O}_v}_E). 
\]
Note that since $\left(\bm{\pi}^{\mathbb{A}^1}_n({\op{B}}\op{SL}_n)\right)_{-n-1}(E)\cong\mathbb{Z}/n!\mathbb{Z}$, we can omit the orientation data in the target, making the above residue map independent of a choice of uniformizer of $v$. To actually describe the residue map, recall that the source has been interpreted in Lemma~\ref{lem:one} above as rank $n$ vector bundles on $\op{Q}_{2n,F}$. The target can similarly be interpreted as rank $n$ vector bundles on $\op{Q}_{2n+1,E}$, cf. \cite{AsokFaselSpheres}. We need to explain how to ``compute'' this residue. Recall that by the definition of contraction we have for any field $E$
\[
\left(\bm{\pi}^{\mathbb{A}^1}_n({\op{B}}\op{SL}_n)\right)_{-n-1}(E)= \op{ker}\left(\left(\bm{\pi}^{\mathbb{A}^1}_n({\op{B}}\op{SL}_n)\right)_{-n}(\mathbb{G}_{\op{m},E})\to  \left(\bm{\pi}^{\mathbb{A}^1}_n({\op{B}}\op{SL}_n)\right)_{-n}(E)\right)
\]
where the map is induced by the unit section. Suppose we have a rank $n$ vector bundle over $\op{Q}_{2n,E}\times\mathbb{G}_{\op{m},E}$ which is trivial over the fiber $\op{Q}_{2n,E}\times\{1\}$. The vector bundle is necessarily trivial over $\{x\}\times\mathbb{G}_{\op{m},E}$ for any choice of base point $x\in \op{Q}_{2n,E}$, hence it is trivial over $\op{Q}_{2n,E}\vee\mathbb{G}_{\op{m},E}$ and therefore descends to a rank $n$ vector bundle over $\op{Q}_{2n,E}\wedge\mathbb{G}_{\op{m}}\simeq \op{Q}_{2n+1,E}$. Therefore, the residue of a vector bundle over $\op{Q}_{2n,E}\times\mathbb{G}_{\op{m},E}$ is computed by first using the group structure to add a vector bundle of the form $p^\ast\mathcal{E}$ for $p\colon\op{Q}_{2n,E}\times\mathbb{G}_{\op{m},E}\to\op{Q}_{2n,E}$ to make the vector bundle trivial over $\op{Q}_{2n,E}\times\{1\}$, and then the residue is the induced vector bundle over $\op{Q}_{2n+1,E}$ as above. The same works if we have a field $F$ with discrete valuation $v$ and residue field $E$: Nisnevich-locally, we can replace $\mathcal{O}_{F,v}\setminus\op{Spec}E$ by $\mathbb{G}_{\op{m},E}$ and do the same construction.\footnote{This is the place where the residue would depend on the choice of uniformizer which influences the choice of identification with $\mathbb{G}_{\op{m}}$; but in our special situation the target has trivial unit action and so the residue is in fact independent of such choice.}

From the above description of the residue, we now deduce the claim. Let $F$ be a field with discrete valuation $v$, valuation ring $\mathcal{O}_v$ and residue field $E$, and $\sigma\in \left(\bm{\pi}^{\mathbb{A}^1}_n({\op{B}}\op{SL}_n)\right)_{-n}(F)$ be a class, represented by a rank $n$ vector bundle over $\op{Q}_{2n}\times_kF$. If the bundle extends (as a rank $n$ vector bundle) to $\op{Q}_{2n}\times_k\mathcal{O}_v$ then its residue is trivial since in this case it is (locally in the Nisnevich topology on $\op{Spec}\mathcal{O}_v$) extended from a constant vector bundle on $\op{Q}_{2n,E}\times\mathbb{G}_{\op{m},E}$. This is exactly the case for the vector bundle from Remark~\ref{rem:vblift}.
\end{proof}

\begin{proofof}{Proposition~\ref{prop:cohop}}
We make use of the fact that the natural map $\mathbf{K}^{\op{M}}_{n+1}/n!\to\mathbf{S}_n$ is an isomorphism after $(n-2)$-fold delooping, tacitly behaving like the first term in the exact sequence of sheaves is $\mathbf{K}^{\op{M}}_{n+1}/n!$. 

We first recall the definition of the boundary map in the long exact cohomology sequence: the short exact sequence of sheaves induces a short exact sequence of Gersten complexes whose relevant part (the two final degrees $n$ and $n+1$) is the following 
\[
\xymatrix{
  \bigoplus_{y\in X^{(n)}} \mathbf{K}^{\op{M}}_{1}(\kappa(y))/n! \ar@{^{(}->}[r] \ar[d] & \bigoplus_{y\in X^{(n)}}\left(\bm{\pi}^{\mathbb{A}^1}_n({\op{B}}\op{SL}_n)\right)_{-n}(\kappa(y);\Lambda^X_y) \ar@{>>}[r] \ar[d]^\partial & \bigoplus_{y\in X^{(n)}}\mathbb{Z} \ar[d] \\
 \bigoplus_{z\in X^{(n+1)}} \mathbb{Z}/n!\mathbb{Z} \ar[r]_\cong &  \bigoplus_{z\in X^{(n+1)}} \mathbb{Z}/n!\mathbb{Z} \ar[r] & 0
}
\]
The class in $\op{CH}^n(X)\cong \op{H}^n(X,\mathbf{K}^{\op{Q}}_n)$ is represented by an element in the right-most term $\bigoplus_{y\in X^{(n)}}\mathbb{Z}$; in our specific case where $X$ has dimension $n+1$, it is a $\mathbb{Z}$-linear combination of curves in $X$. We lift this cycle to the upper middle term, apply the boundary map $\partial$, and the result is a representative for the cohomology class in $\op{CH}^{n+1}(X)/n!\cong\op{H}^{n+1}(X,\mathbf{S}_{n+1})$. Note that the sheaves $\mathbf{K}^{\op{M}}_{n+1}$ and $\mathbf{K}^{\op{Q}}_n$ are orientable\footnote{Orientable here means that the action of the units on the contractions is trivial. This implies in particular that the $\mathbf{K}^{\op{MW}}_0$-module structure extending the unit action factors through the dimension homomorphism $\op{GW}(F)\to \mathbb{Z}$.}, hence we can drop the orientation information $\Lambda^X_y$ in the outer terms of the first line; but we need to include orientation information $\Lambda^X_y$ in the middle term since the sheaf $\bm{\pi}^{\mathbb{A}^1}_n({\op{B}}\op{SL}_n)$ is not orientable. 

To discuss the lift of a class in $\op{CH}^n(X)$ to an element of the middle term 
\[
\bigoplus_{y\in X^{(n)}}\left(\bm{\pi}^{\mathbb{A}^1}_n({\op{B}}\op{SL}_n)\right)_{-n}(\kappa(y))\otimes_{\mathbb{Z}[\kappa(y)^\times]}\mathbb{Z}[(\Lambda^X_y)^\times], 
\]
we employ Lemma~\ref{lem:one}. The lemma provides us with a vector bundle $\mathcal{V}\to\op{Q}_{2n,k}$ (over the base field $k$) whose pullback to any extension field $F/k$ stabilizes to a generator of $\tilde{\op{K}}_0(\op{Q}_{2n,F})$. For a $\mathbb{Z}$-linear combination $\sum_im_i[\op{C}_i]$ of curves $C_i\hookrightarrow X$, we can then provide an explicit lift by 
\[
\sum_i m_i[\mathcal{V}_i/C_i]\otimes \sigma_i,
\]
where $[\mathcal{V}_i/C_i]$ is the class in $\left(\bm{\pi}^{\mathbb{A}^1}_n{\op{B}}\op{SL}_n\right)_{-n}(\kappa(C_i))$ of the pullback of the above vector bundle $\mathcal{V}\to\op{Q}_{2n,k}$ to $\op{Q}_{2n}\times C_i$, and $\sigma_i\in \Lambda^X_{\kappa(C_i)}$ is a non-zero element. Note that $\sigma_i\in \Lambda^X_{\kappa(C_i)}$ can be considered as a rational section of the relative normal bundle $\omega_{\tilde{C_i}/X}$ of the composition $\tilde{C_i}\to C_i\hookrightarrow X$ of the normalization of the curve $C_i$ followed by the embedding of $C_i$ into $X$ (using that the normalization is an isomorphism away from the ramification points, hence we have a natural identification of the relative normal bundles there).

In the next step of the construction of the connecting homomorphism $\op{CH}^n(X)\to \op{CH}^{n+1}(X)/n!$ for the exact sequence of sheaves, we have to compute the image of the above lift under the differential $\partial$ in the Gersten complex for $\bm{\pi}^{\mathbb{A}^1}_n{\op{B}}\op{SL}_n$. For this, recall from Section~\ref{sec:rscomplex} the definition of the differential 
\[
\partial^{\kappa(C)}\colon \left(\bm{\pi}^{\mathbb{A}^1}_n{\op{B}}\op{SL}_n\right)_{-n}(\kappa(C))\otimes_{\mathbb{Z}[\kappa(C)^\times]}\mathbb{Z}[(\Lambda^X_C)^\times]\to \bigoplus_{z\in X^{(n+1)}, z\in C} \mathbb{Z}/n!\mathbb{Z},
\]
where we have already specialized to the relevant situation where $X$ is $n+1$-dimensional and $C\hookrightarrow X$ is a curve in $X$, and the orientation data is omitted in the target groups. For computing the differential $\partial^{\kappa(C)}_z$ we have to consider the normalization $\tilde{C}\to C$ of the curve, compute the residue map 
\[
\partial_{\tilde{z}_i}^{\kappa(C)}\colon \left(\bm{\pi}^{\mathbb{A}^1}_n {\op{B}}\op{SL}_n\right)_{-n}(\kappa(C)) \otimes_{\mathbb{Z}[\kappa(C)^\times]} \mathbb{Z}[(\Lambda^X_C)^\times]\to\mathbb{Z}/n!\mathbb{Z}
\]
for all the points $\tilde{z}_i$ on $\tilde{C}$ lying over $z$ and then compose that with the absolute transfer map $\bigoplus_{\tilde{z}_i/z}\mathbb{Z}/n!\mathbb{Z}\to \mathbb{Z}/n!\mathbb{Z}$. Since $\mathbf{K}^{\op{M}}_{n+1}$ is orientable (or differently, since the unit action on  $\mathbb{Z}/n!\mathbb{Z}$ is trivial), the absolute transfer map reduces to the transfer map induced from Milnor K-theory, hence it takes an element $\bigoplus_i m_i\in\bigoplus_{\tilde{z}_i}\mathbb{Z}/n!\mathbb{Z}$ to $\sum_i[\kappa(\tilde{z}_i):\kappa(z)]m_i\in\mathbb{Z}/n!\mathbb{Z}$. 

Now for the computation of the residue of the above lift. As a first step, Lemma~\ref{lem:three} describes the computation of the residue of the lifted vector bundle before twisting. Then the actual residue map
\[
\partial_{\tilde{z}_i}^{\kappa(C)}\colon \left(\bm{\pi}^{\mathbb{A}^1}_n {\op{B}}\op{SL}_n\right)_{-n}(\kappa(C)) \otimes_{\mathbb{Z}[\kappa(C)^\times]} \mathbb{Z}[(\Lambda^X_{\kappa(C)})^\times]\to\mathbb{Z}/n!\mathbb{Z}
\]
is obtained by twisting the residue map described in Lemma~\ref{lem:three} by the relative normal bundle $\omega_{\tilde{C}/X}$. More concretely, let  $\mathcal{V}\otimes \sigma\in \left(\bm{\pi}^{\mathbb{A}^1}_n {\op{B}}\op{SL}_n\right)_{-n}(\kappa(C);\Lambda^X_{\kappa(C)})$ be given by a vector bundle $\mathcal{V}\to\op{Q}_{2n,\kappa(C)}$ and a rational section of $\omega_{\tilde{C}/X}$. We choose a uniformizer $\pi$ for the local ring of $\tilde{z}_i$ on $\tilde{C}$. There exists an integer $n\in \mathbb{Z}$ such that the section $\pi^{-n}\sigma$ is invertible at $\tilde{z}_i$, and 
\[
\langle\pi^n\rangle\mathcal{V}\otimes \pi^{-n}\sigma=\mathcal{V}\otimes \sigma\in \left(\bm{\pi}^{\mathbb{A}^1}_n {\op{B}}\op{SL}_n\right)_{-n}(\kappa(C)) \otimes_{\mathbb{Z}[\kappa(C)^\times]} \mathbb{Z}[(\Lambda^X_{\kappa(C)})^\times].
\]
Here $\langle\pi^n\rangle\mathcal{V}$ is the result of the action of the unit group $\kappa(C)^\times$ on the element $\mathcal{V}\in (\bm{\pi}^{\mathbb{A}^1}_n{\op{B}}\op{SL}_n)_{-n}(\kappa(C))$. Then $\partial_{\tilde{z}_i}^{\kappa(C)}(\mathcal{V}\otimes \sigma)$ is given by the residue of $\langle \pi^n\rangle \mathcal{V}$ for this choice of uniformizer. By Lemma~\ref{lem:three} the untwisted residue of $\mathcal{V}$ is $0$, and by Lemma~\ref{lem:two}, the residue of $\partial^{\kappa(C)}_{\tilde{z}_i}(\mathcal{V}\otimes \sigma)$ is 2-torsion. Finally, the absolute transfer is the sum of multiples of such 2-torsion elements, hence 2-torsion. This implies the claim.
\end{proofof}

\subsection{Construction of stably free modules}
\label{sec:existence}

We now describe how to construct the stably free modules. The standard stabilization morphism 
\[
\op{SL}_n\to\op{SL}_{n+1}\colon M\mapsto \left(\begin{array}{cc}
M & 0 \\ 0 & 1\end{array}\right)
\]
induces a morphism ${\op{B}}\op{SL}_n\to{\op{B}}\op{SL}_{n+1}$ which maps an oriented projective $R$-module $\mathscr{P}$ of rank $n$ to $\mathscr{P}\oplus R$ with the induced orientation coming from $\bigwedge^{n+1}(\mathscr{P}\oplus R)\cong \bigwedge^n(\mathscr{P})\otimes R$. There is an $\mathbb{A}^1$-fiber sequence $\op{Q}_{2n+1}\to{\op{B}}\op{SL}_n\to{\op{B}}\op{SL}_{n+1}$, and the above exact sequence arises from the associated long exact sequence of $\mathbb{A}^1$-homotopy sheaves. In particular, the induced morphism $\bm{\pi}^{\mathbb{A}^1}_n(\op{Q}_{2n+1})\to \bm{\pi}^{\mathbb{A}^1}_n({\op{B}}\op{SL}_n)$ factors as
\[
\mathbf{K}^{\op{MW}}_{n+1}\cong \bm{\pi}^{\mathbb{A}^1}_n(\op{Q}_{2n+1})\twoheadrightarrow \mathbf{K}^{\op{M}}_{n+1}\twoheadrightarrow \mathbf{S}_{n+1}\hookrightarrow \bm{\pi}^{\mathbb{A}^1}_n({\op{B}}\op{SL}_n).
\]

To construct a projective module of rank $n$ which becomes trivial upon addition of a free rank one summand, we can proceed as follows. Suppose we have a morphism $\alpha\colon X\to\op{Q}_{2n+1}$ corresponding to a unimodular row of length $n$, we can consider the composition $X\xrightarrow{\alpha}\op{Q}_{2n+1}\to{\op{B}}\op{SL}_n$. If we can show that the image of $\alpha$ under the composition
\[
[X,\op{Q}_{2n+1}]_{\mathbb{A}^1}\to \op{H}^n(X,\mathbf{K}^{\op{MW}}_{n+1}) \to \op{H}^n(X,\bm{\pi}^{\mathbb{A}^1}_n({\op{B}}\op{SL}_n))
\]
is non-trivial, we will obtain a non-trivial stably free module, as required. This procedure was suggested by Aravind Asok to replace an earlier argument which didn't properly address the non-uniqueness of lifting classes.

The non-triviality of the image of $\alpha$ in  $\op{H}^n_{\op{Nis}}(X,\bm{\pi}^{\mathbb{A}^1}_n({\op{B}}\op{SL}_n))$ can now be discussed by means of the above exact sequence describing $\bm{\pi}^{\mathbb{A}^1}_n({\op{B}}\op{SL}_n)$. It induces a long exact sequence of Nisnevich cohomology groups whose relevant portion is
\[
 \op{H}^{n-1}_{\op{Nis}}(X,\mathbf{K}^{\op{Q}}_n)\to \op{H}^n_{\op{Nis}}(X,\mathbf{K}^{\op{M}}_{n+1}/n!) \to \op{H}^n_{\op{Nis}}(X,\bm{\pi}^{\mathbb{A}^1}_n({\op{B}}\op{SL}_n))\to \op{H}^n_{\op{Nis}}(X,\mathbf{K}^{\op{Q}}_n)
\]
By construction, the image of $\alpha$ lands in the image of $\op{H}^n(X,\mathbf{K}^{\op{M}}_{n+1}/n!)$. To get non-triviality of the image of $\alpha$ in $\op{H}^n_{\op{Nis}}(X,\bm{\pi}^{\mathbb{A}^1}_n({\op{B}}\op{SL}_n))$, 
the key point is the injectivity up to 2-torsion of the morphism $\op{H}^n(X,\mathbf{K}^{\op{M}}_{n+1}/n!)\to \op{H}^n(X,\bm{\pi}^{\mathbb{A}^1}_n({\op{B}}\op{SL}_n))$ induced from the above exact sequence defining $\bm{\pi}^{\mathbb{A}^1}_n({\op{B}}\op{SL}_n)$. 
The injectivity up to 2-torsion of the morphism $\op{H}^n(X,\mathbf{K}^{\op{M}}_{n+1}/n!)\to \op{H}^n(X,\bm{\pi}^{\mathbb{A}^1}_n({\op{B}}\op{SL}_n))$ is a nontrivial statement which will follow from the discussion of the cohomology operation in Proposition~\ref{prop:cohop}. This is the content of the following theorem:

\begin{theorem}
\label{thm:mkalt}
Let $F$ be an algebraically closed field and set $k=F(T)$. Let $X$ be an $n+1$-dimensional smooth variety over $k$ with a covering $X=Y\cup Z$ such that $Y\cap Z$ is affine. Assume there is a class $c\in\op{CH}^{n+1}(X)/n!$ which restricts trivially to both $Y$ and $Z$ and whose reduction in $\op{CH}^{n+1}(X)/n$ is non-trivial. Then there exists a non-trivial stably free module over $Y\cap Z$ whose lifting class maps to $c$ under the boundary map in the Mayer--Vietoris sequence. 
\end{theorem}

\begin{proof}
By Proposition~\ref{prop:mkclass} resp. its corollary, the geometric assumptions imply that there is a surjection 
\[
[Y\cap Z,\op{Q}_{2n+1}]_{\mathbb{A}^1}\to \op{H}^n_{\op{Nis}}(Y\cap Z,\mathbf{K}^{\op{MW}}_{n+1})\twoheadrightarrow \op{CH}^{n+1}(X)/n!
\]
We call $\alpha$ any map $Y\cap Z\to\op{Q}_{2n+1}$ lifting the class $c$. This map will necessarily induce a non-trivial element in $\op{H}^n_{\op{Nis}}(Y\cap Z,\mathbf{K}^{\op{M}}_{n+1}/n!)$. The canonical epimorphism $\mathbf{K}^{\op{M}}_{n+1}/n!\to\mathbf{S}_{n+1}$ induces isomorphisms after $n-1$-fold contraction. This implies that it induces an isomorphism 
\[
\op{H}^n_{\op{Nis}}(Y\cap Z,\mathbf{K}^{\op{M}}_{n+1}/n!)\cong\op{H}^n_{\op{Nis}}(Y\cap Z,\mathbf{S}_{n+1}).
\]
We need to prove that the class has non--trivial image under the morphism 
\[
\op{H}^n_{\op{Nis}}(Y\cap Z,\mathbf{K}^{\op{M}}_{n+1}/n!) \to \op{H}^n_{\op{Nis}}(Y\cap Z,\bm{\pi}^{\mathbb{A}^1}_n({\op{B}}\op{SL}_n))
\]
To do this, consider the ladder of Mayer--Vietoris sequences (associated to the covering $X=Y\cup Z$) for cohomology with coefficients in $\mathbf{K}^{\op{Q}}_p$ and $\mathbf{K}^{\op{M}}_{p+1}/p!$, connected by the boundary map arising from the exact sequence
\[
0\to\mathbf{S}_{n+1}\to\bm{\pi}^{\mathbb{A}^1}_n({\op{B}}\op{SL}_n)\to \mathbf{K}^{\op{Q}}_n\to 0. 
\]
This ladder contains the following square which is commutative up to sign
\[
\xymatrix{
\op{H}^{n-1}_{\op{Nis}}(Y\cap Z,\mathbf{K}^{\op{Q}}_n) \ar[r] \ar[d]_\partial & \op{CH}^n(X) \ar[d]^\partial \\
\op{H}^n_{\op{Nis}}(Y\cap Z,\mathbf{K}^{\op{M}}_{n+1}/n!) \ar[r] & \op{CH}^{n+1}(X)/n!
}
\]
where the horizontal maps are the boundary maps in the respective Mayer--Vietoris sequences, and the vertical maps are the boundary maps for the above exact sequence of sheaves. Assuming that the image of $\alpha$ in $\op{H}^n_{\op{Nis}}(Y\cap Z,\mathbf{K}^{\op{M}}_{n+1}/n!)$ is in the image of $\partial$, then its associated class in $\op{CH}^{n+1}(X)/n!$ is also in the image of $\partial$. But by Proposition~\ref{prop:cohop}, we know that the composition of the boundary map $\op{CH}^n(X)\to\op{CH}^{n+1}(X)/n!$ with the reduction $\op{CH}^{n+1}(X)/n!\to\op{CH}^{n+1}(X)/n$ is trivial. On the other hand, the class $c$ in $\op{CH}^{n+1}(X)/n!$ has nontrivial image in $\op{CH}^{n+1}(X)/n$ and hence cannot lie in the image of $\partial$. We have thus proved that the image of $\alpha$ in $\op{H}^n_{\op{Nis}}(Y\cap Z,\mathbf{S}_{n+1})$ obtained from lifting a non-trivial class $c\in\op{CH}^{n+1}(X)/n!$ has non-trivial image under the map
\[
\op{H}^n_{\op{Nis}}(Y\cap Z,\mathbf{S}_{n+1})\to \op{H}^n_{\op{Nis}}(Y\cap Z,\bm{\pi}^{\mathbb{A}^1}_n({\op{B}}\op{SL}_n))
\]

As mentioned in Section~\ref{sec:comparison}, the element in $[Y\cap Z,\op{Q}_{2n+1}]_{\mathbb{A}^1}$ corresponds (by results discussed in \cite[\S 4]{AsokFaselSpheres}) to an algebraic map $Y\cap Z\to\op{Q}_{2n+1}$ such that the composition $Y\cap Z\to\op{Q}_{2n+1}\to\mathbb{A}^{n+1}\setminus\{0\}$ is a unimodular row of length $n+1$ over the coordinate ring of $Y\cap Z$. Since the image of the class $\alpha$ in $\op{H}^n_{\op{Nis}}(Y\cap Z,\bm{\pi}^{\mathbb{A}^1}_n({\op{B}}\op{SL}_n))$ is non-trivial, this unimodular row gives rise to a non-trivial stably free module as claimed.
\end{proof}

\begin{remark}
The above theorem now explains Mohan Kumar's examples of stably free modules. Over a base field $K$, take $X\subset\mathbb{P}^{n+1}$ to be the complement of a hypersurface such that $\op{CH}^{n+1}(X)/n!$ is nontrivial. Now take any two hypersurfaces $F_1,F_2\in\mathbb{P}^n$ which intersect only outside $X$ and which each have a $K$-rational point. Then $Y=X\setminus F_1$ and $Z=X\setminus F_2$ provides a covering $X=Y\cup Z$ where $\op{CH}^{n+1}(Y)=\op{CH}^{n+1}(Z)=0$, as required in Theorem~\ref{thm:mkalt}. The fact that such a situation can be arranged over $K=k(T)$ with $k$ algebraically closed and with $n=p$ a prime is proved in \cite{mohan:kumar}, cf. the discussion in Section~\ref{sec:mkgeometry}. Clearing denominators this shows the existence of rank $p$ stably free modules over varieties of dimension $p+2$ over algebraically closed fields.
\end{remark}

\begin{remark}
The same argument actually provides several non-isomorphic examples of stably free modules. In the situation of Mohan Kumar's examples, $X\subseteq\mathbb{P}^{p+1}$ has $\op{CH}^{p+1}(X)\cong\mathbb{Z}/p\mathbb{Z}$. There are thus $p-1$ different non-zero elements in $\op{CH}^{p+1}(X)$. We can choose preimages for each of those elements under the surjective composition
\[
[Y\cap Z, \op{Q}_{2p+1}]_{\mathbb{A}^1}\to\op{H}^n(Y\cap Z,\mathbf{K}^{\op{M}}_{p+1}/p!)\twoheadrightarrow\op{CH}^{p+1}(X)/p.
\]
The arguments in Theorem~\ref{thm:mkalt} then provide $p-1$ stably free modules. By construction these stably free modules will have distinct images under the morphism $[Y\cap Z, \op{Q}_{2p+1}]_{\mathbb{A}^1}\to\op{H}^n(Y\cap Z,\mathbf{K}^{\op{M}}_{p+1}/p!)$. Applying the arguments of Theorem~\ref{thm:mkalt} to the (nonzero) differences of these classes shows that their images under the homomorphism 
\[
\op{H}^p(Y\cap Z,\mathbf{K}^{\op{M}}_{p+1}/p!)\to \op{H}^n(Y\cap Z,\bm{\pi}^{\mathbb{A}^1}_p({\op{B}}\op{SL}_p))
\]
are still distinct. This implies that the $p-1$ stably free modules as above are pairwise non-isomorphic. 
\end{remark}

\section{Stably free modules: Mohan Kumar's examples at the prime 2}
\label{sec:mk2}

Now we discuss Mohan Kumar's stably free modules of rank $2$. While the constructions in \cite{mohan:kumar} work for odd and even primes, the $\mathbb{A}^1$-topological reinterpretation here is slightly different for the prime $2$, due to a difference in the structure of the relevant $\mathbb{A}^1$-homotopy sheaf. The goal is to get examples of a smooth variety of dimension $3$ supporting a rank $2$ stably free module detected by $\op{H}^2_{\op{Nis}}(X,\bm{\pi}^{\mathbb{A}^1}_2{\op{B}}\op{SL}_2)$. Note that the stabilization morphism 
\[
\mathbf{K}^{\op{MW}}_2\cong \bm{\pi}^{\mathbb{A}^1}_2{\op{B}}\op{SL}_2 \to \bm{\pi}^{\mathbb{A}^1}_2{\op{B}}\op{SL}_3\cong\mathbf{K}^{\op{M}}_2
\]
induced from the stabilization morphism ${\op{B}}\op{SL}_2\to{\op{B}}\op{SL}_3$ is the natural projection, cf. \cite[Remark 7.21]{MField}. Note also that the lifting class for a projective module $\mathscr{P}$ naturally lives in $\op{H}^2_{\op{Nis}}(X,\bm{\pi}^{\mathbb{A}^1}_2{\op{B}}\op{SL}_2(\det\mathscr{P}))\cong \widetilde{\op{CH}}^2(X,\det\mathscr{P})$. Since we are looking for stably free modules, we are really interested in Chow--Witt groups with the trivial duality. Summing up, the stably free modules of rank $2$ should be detected by a non-trivial class in 
\[
\ker\left(\widetilde{\op{CH}}^2(X)\cong\op{H}^2_{\op{Nis}}(X,\mathbf{K}^{\op{MW}}_2)\to \op{H}^2_{\op{Nis}}(X,\mathbf{K}^{\op{M}}_2)\cong \op{CH}^2(X)\right).
\]

From the discussion in Section~\ref{sec:mkgeometry}, we obtain a smooth affine variety $Y\cap Z$ together with a morphism $\alpha\colon Y\cap Z\to\op{Q}_5$ which is detected on $\op{H}^2(Y\cap Z,\mathbf{K}^{\op{MW}}_3)$. We want to show that the composition $Y\cap Z\to\op{Q}_5\to{\op{B}}\op{SL}_2$  of $\alpha$ with the inclusion of the homotopy fiber of the stabilization morphism is not null-$\mathbb{A}^1$-homotopic. 

\begin{lemma}
\label{lem:xi3}
Let $k$ be an algebraically closed field of characteristic $\neq 2$ and let $X$ be the open subvariety of $\mathbb{P}^3$ over $k(T)$ given by Mohan Kumar's construction, cf.~Section~\ref{sec:mkgeometry}. There is a non-trivial class in $\op{H}^3_{\op{Nis}}(X,\mathbf{I}^3)$, detected by $\op{CH}^3(X)/2$.
\end{lemma}

\begin{proof}
By the results of Mohan Kumar, $\op{CH}^3(X)/2\cong\op{H}^3_{\op{Nis}}(X,\mathbf{K}^{\op{M}}_3/2)\cong\mathbb{Z}/2\mathbb{Z}$ is non-trivial. The exact sequence
\[
0\to\mathbf{I}^4\to\mathbf{I}^3\to\mathbf{K}^{\op{M}}_3/2 \to 0
\]
arising from Voevodsky's solution of the Milnor conjecture induces a sequence
\[
\op{H}^3_{\op{Nis}}(X,\mathbf{I}^3)\to \op{H}^3_{\op{Nis}}(X,\mathbf{K}^{\op{M}}_3/2)\to \op{H}^4_{\op{Nis}}(X,\mathbf{I}^4).
\] 
The last group vanishes since the Nisnevich cohomological dimension of $X$ is $3$. Therefore, we get a surjection $\op{H}^3_{\op{Nis}}(X,\mathbf{I}^3)\to\mathbb{Z}/2\mathbb{Z}$.
\end{proof}

Now we can consider the Mayer--Vietoris sequence again to lift this non-trivial element to another $\mathbf{I}^3$-cohomology group. Since $X$ is an open subvariety of $\mathbb{P}^3$ with closed hypersurface complement $S=\mathbb{P}^3\setminus X$ (with reduced induced subscheme structure), we have a localization sequence
\[
\op{H}^3(\mathbb{P}^3,\mathbf{I}^3)\to \op{H}^3(X,\mathbf{I}^3)\to \op{H}^3(S,\mathbf{I}^2)=0.
\]
The triviality of the last group follows from dimension reasons. By the computations in \cite{fasel:ij}, we have $\op{H}^3(\mathbb{P}^3,\mathbf{I}^3)\cong\op{W}(k(T))$. In particular, we can choose a nontrivial class in $\op{H}^3_{\op{Nis}}(X,\mathbf{I}^3)$ to be represented by an element of $\op{W}(k)$ of odd rank, supported on a rational point. Because the hypersurface complements of $Y$ and $Z$ contain rational points, this choice of class in $\mathbf{I}^3$-cohomology of $X$ will trivialize in the covering. This implies the following statement:

\begin{proposition}
\label{prop:h2i3}
Let $k$ be an algebraically closed field of characteristic $\neq 2$. In the situation $p=2$ of Mohan Kumar's example, the boundary map of the Mayer--Vietoris sequence associated to the covering $X=Y\cup Z$ provides a surjection
\[
\op{H}^2_{\op{Nis}}(Y\cap Z,\mathbf{I}^3)\twoheadrightarrow \op{H}^3_{\op{Nis}}(X,\mathbf{I}^3)\twoheadrightarrow \op{CH}^3(X)/2\cong\mathbb{Z}/2\mathbb{Z}.
\]
\end{proposition}

This produces a class in $\op{H}^2_{\op{Nis}}(Y\cap Z,\mathbf{I}^3)$ which is detected by the non-trivial $\mathbf{I}^3$-cohomology of $X$. Now consider the exact sequence
\[
0\to \mathbf{I}^3\to\mathbf{K}^{\op{MW}}_2\to\mathbf{K}^{\op{M}}_2\to 0. 
\]
We want to show that the class produced above has non-trivial image under 
\[
\op{H}^2_{\op{Nis}}(Y\cap Z,\mathbf{I}^3)\to\op{H}^2_{\op{Nis}}(Y\cap Z,\mathbf{K}^{\op{MW}}_2)\cong\widetilde{\op{CH}}^2(Y\cap Z). 
\]
This would produce a non-trivial element in the kernel of the natural projection map $\widetilde{\op{CH}}^2(Y\cap Z)\to\op{CH}^2(Y\cap Z)$. As before, we use the  Mayer--Vietoris sequences for the covering $X=Y\cup Z$. We use the two sequences for coefficients with $\mathbf{K}^{\op{M}}_2$ and $\mathbf{I}^3$, connected by the boundary map associated to the above exact sequence. This ladder of Mayer--Vietoris sequences contains the  square which is commutative up to sign:
\[
\xymatrix{
\op{H}^1_{\op{Nis}}(Y\cap Z,\mathbf{K}^{\op{M}}_2) \ar[r] \ar[d] & \op{H}^2_{\op{Nis}}(X,\mathbf{K}^{\op{M}}_2) \ar[d] \\
\op{H}^2_{\op{Nis}}(Y\cap Z,\mathbf{I}^3)\ar[r] & \op{H}^3_{\op{Nis}}(X,\mathbf{I}^3).
}
\]
The horizontal arrows are the boundary maps for the Mayer--Vietoris sequences, the vertical maps are the boundary maps for the exact sequence of sheaves. To check that the left-hand vertical map doesn't hit the element from Proposition~\ref{prop:h2i3} it suffices to show that the right-hand vertical map doesn't hit the non-trivial $2$-torsion from Lemma~\ref{lem:xi3}.  So we need to compute the right-hand boundary map, which can be viewed as an integral Bockstein operation. By \cite[Proposition 11.6]{fasel:ij}, we know that the Bockstein map $\beta\colon\op{CH}^2(\mathbb{P}^3)\to\op{H}^3_{\op{Nis}}(\mathbb{P}^3,\mathbf{I}^3)$ is trivial. Since the Bockstein map is compatible with the maps in localization sequences, we get a commutative square
\[
\xymatrix{
\op{CH}^2(\mathbb{P}^3) \ar[r] \ar[d]_\beta & \op{CH}^2(X) \ar[d]^\beta \\
\op{H}^3(\mathbb{P}^3,\mathbf{I}^3) \ar[r]  & \op{H}^3(X,\mathbf{I}^3)
}
\]
where the horizontal arrows are the restriction to open subschemes in the localization sequence. Note that Zariski and Nisnevich cohomology coincide. By the localization sequence, the restriction $\op{CH}^2(\mathbb{P}^3)\to\op{CH}^2(X)$ is in fact surjective. This implies, in particular, that the right-hand vertical map $\beta$ is the zero map, because the left-hand vertical map is trivial, as claimed. 
The previous square from the ladder of Mayer-Vietoris sequences now implies that the element produced in Proposition~\ref{prop:h2i3} induces a non-trivial class in 
\[
\ker\left(\widetilde{\op{CH}}^2(Y\cap Z)\to\op{CH}^2(Y\cap Z)\right).
\]

We can take any such non-trivial element as a class in $\op{H}^2_{\op{Nis}}(Y\cap Z,\mathbf{K}^{\op{MW}}_2)$ and lift it along the surjection $[Y\cap Z,\op{Q}_5]_{\mathbb{A}^1}\twoheadrightarrow \op{H}^2_{\op{Nis}}(Y\cap Z,\mathbf{K}^{\op{MW}}_2)$. This corresponds to a unimodular row of length $3$ which gives rise to a stably free module of rank $2$. By construction, its image in  $\op{H}^2_{\op{Nis}}(Y\cap Z,\bm{\pi}_2^{\mathbb{A}^1}{\op{B}}\op{SL}_2)$ is non-trivial. Applying the representability theorem~\ref{thm:representability}, we have then proved the following:

\begin{theorem}
Let $k$ be an algebraically closed field of characteristic $\neq 2$. There exists a $3$-dimensional smooth affine scheme $Y\cap Z$ over $k(T)$ and a non-trivial stably trivial rank $2$ vector bundle on $Y\cap Z$. Clearing denominators, there exists a $4$-dimensional smooth affine scheme over $k$ with a non-trivial stably free rank $2$ bundle over it. 
\end{theorem}

\begin{remark}
Over other fields $k$ it could be possible to get other types of examples. Since the hypersurface complement of $X$ has no rational points, the map $\op{H}^3_{\op{Nis},\mathbb{P}^3\setminus X}(\mathbb{P}^3,\mathbf{I}^3)\to \op{H}^3_{\op{Nis}}(\mathbb{P}^3,\mathbf{I}^3)$ should map a class (given by elements of $\op{W}(k(x))$ supported on points $x$ on $W$) to the sum of transfers of the classes from $\op{W}(k(x))$ supported on a rational point. In particular, the ideal generated by transfers of classes from points on $W$ should be strictly smaller than the fundamental ideal whenever there are non-trivial quaternion algebras over $k$. Classes not in the image of transfer should then yield non-trivial classes in $\op{H}^3_{\op{Nis}}(X,\mathbf{I}^3)$ not detected on $\op{CH}^3(X)/2$. Applying the previous argument would produce stably free modules of a more arithmetic nature on $X$. This could be quite similar to the examples discussed in \cite[Section 3]{bhatwadekar:fasel:sane}.
\end{remark}

\section{Stably trivial torsors for other groups}
\label{sec:rest}

In this section, we discuss analogous results for torsors under other groups. In slight variation of the notion of stably trivial torsors we can consider, for an embedding $H\to G$ of algebraic groups, $H$-torsors which become trivial after extending the structure group to $G$. We discuss such phenomena in the other classical series $B_n$, $C_n$ and $D_n$, and for the homomorphism $\op{G}_2\to\op{Spin}(7)$. 

\subsection{Symplectic groups}

We first discuss how the construction of Mohan Kumar also gives rise to stably free torsors for the symplectic groups. By  \cite[Section 4.2]{AsokFaselSplitting}, there is a cartesian square of algebraic groups
\[
\xymatrix{
\op{Sp}_{2n}\ar[r] \ar[d] & \op{Sp}_{2n+2} \ar[d] \\
\op{SL}_{2n+1}\ar[r] & \op{SL}_{2n+2}.
}
\]
This implies the existence of a commutative diagram of $\mathbb{A}^1$-fiber sequences 
\[
\xymatrix{
\mathbb{A}^{2n+2}\setminus\{0\} \ar[r] \ar[d]_\cong &
{\op{B}}\op{Sp}_{2n}\ar[r] \ar[d] & {\op{B}}\op{Sp}_{2n+2} \ar[d] \\
\mathbb{A}^{2n+2}\setminus\{0\} \ar[r]  & {\op{B}}\op{SL}_{2n+1}\ar[r] & {\op{B}}\op{SL}_{2n+2}.
}
\]

Recall that Section~\ref{sec:mkgeometry} provided a smooth affine variety $Y\cap Z$ with a morphism $Y\cap Z\to\op{Q}_{2p+1}$ which is detected in Milnor K-cohomology. For $p$ odd, setting $2n=p-1$, we can now compose  the morphism $Y\cap Z\to \op{Q}_{2p+1}\to\mathbb{A}^{p+1}\setminus\{0\}$ with the natural map $\mathbb{A}^{p+1}\setminus\{0\}\to{\op{B}}\op{Sp}_{p-1}$ which is the inclusion of the $\mathbb{A}^1$-homotopy fiber of the stabilization map ${\op{B}}\op{Sp}_{p-1}\to {\op{B}}\op{Sp}_{p+1}$ in the above diagram. 

By construction, this produces a symplectic vector bundle which is trivial after adding a symplectic line, by the top horizontal fiber sequence. The middle vertical map takes an $\op{Sp}_{p-1}$-torsor to the direct sum of its underlying vector bundle with a trivial line bundle. The commutativity of the left square in the above diagram of fiber sequences thus implies that Mohan Kumar's stably trivial module splits off a trivial line bundle and the remaining vector bundle of rank $p-1$ has a symplectic structure. Nontriviality of Mohan Kumar's stably trivial modules then implies that the corresponding stably trivial torsors for the symplectic groups are nontrivial. In fact, from the commutative diagram of fiber sequences, we  get a commutative diagram
\[
\xymatrix{
[Y\cap Z,\op{Q}_{2p+1}]_{\mathbb{A}^1} \ar[r] & \op{H}^p(Y\cap Z,\mathbf{K}^{\op{MW}}_{p+1}) \ar[r] \ar[rd] & \op{H}^p(Y\cap Z,\bm{\pi}^{\mathbb{A}^1}_p{\op{B}}\op{Sp}_{p-1}) \ar[d] \\
&&\op{H}^p(Y\cap Z,\bm{\pi}^{\mathbb{A}^1}_p{\op{B}}\op{SL}_{p})
}
\]
from which we see that the stably trivial symplectic modules are in fact detected by nontrivial classes in $\op{H}^p(Y\cap Z,\bm{\pi}^{\mathbb{A}^1}_p{\op{B}}\op{Sp}_{p-1})$. 

\begin{theorem}
\label{thm:sp}
Let $k$ be an algebraically closed field of characteristic $\neq 2$. For every odd prime $p$, there exists a $p+1$-dimensional smooth affine variety over $k(T)$ and a non-trivial stably trivial $\op{Sp}_{p-1}$-torsor over it. Clearing denominators, there exists a $p+2$-dimensional smooth affine variety over $k$ and a stably trivial non-trivial $\op{Sp}_{p-1}$-torsor over it. 
\end{theorem}

\begin{corollary}
Let $k$ be an algebraically closed field of characteristic $\neq 2$ and let $p$ be an odd prime. The stably free modules of rank $p$ of Mohan Kumar split off a trivial line, and the resulting stably free module of rank $p-1$ has a symplectic structure. 
\end{corollary}

As pointed out by one of the referees,  for a unimodular row $\underline{a}:=(a_1,\dots,a_{p+1})$ we have another unimodular row $(-a_2,a_1,-a_4,a_3,\dots,-a_{p+1},a_p)$ which is in the kernel of the surjection defined by $\underline{a}$. This implies that the vector bundle defined by the unimodular row $\underline{a}$ splits off a trivial line bundle such that the remaining vector bundle has a symplectic structure, cf. \cite[\S 3]{fasel:sphere}. 

\subsection{Orthogonal groups}

Via the sporadic isomorphisms, the stably trivial symplectic bundles provide stably trivial torsors for low-rank spin groups. 

The first type of examples arises from the rank 3 stably free modules of Mohan Kumar. As discussed above, these modules split as direct sum of a trivial line bundle and a nontrivial oriented rank 2 bundle. Via the low-dimensional sporadic isomorphism $\op{SL}_2\cong\op{Spin}(3)$, the rank 2 vector bundles correspond to $\op{Spin}(3)$-torsors. These torsors are stably trivial in the sense that they become trivial after stabilization to $\op{Spin}(5)$. In fact, the spin torsors induce non-split quadratic forms which become split after adding a hyperbolic plane. A detailed discussion of these examples and the relevant $\mathbb{A}^1$-homotopy computations can be found in \cite{hyperbolic-dim3}, cf. in particular Example 5.18. 

In a similar way, the low-dimensional sporadic isomorphism $\op{SL}_2\times\op{SL}_2\cong\op{Spin}(4)$, the above rank 2 stably trivial vector bundles also give rise to $\op{Spin}(4)$-torsors which become trivial after stabilization to $\op{Spin}(5)$. 

Finally, there is a third type of stably trivial spin torsors. On the one hand, $\mathbf{K}^{\op{MW}}_2$ appears in the $\mathbb{A}^1$-fundamental groups of $\op{Spin}(n)$ for $3\leq n\leq 5$, and on the other hand, the $\mathbb{A}^1$-fundamental group of $\op{Spin}(n)$ is $\mathbf{K}^{\op{M}}_2$ for $n\geq 6$. Consequently, the rank 2 vector bundles detected by elements in $\ker\left(\widetilde{\op{CH}}^2(X)\to\op{CH}^2(X)\right)$ (as discussed in Section~\ref{sec:mk2}) can be used to construct stably trivial spin torsors. A detailed discussion of these examples can be found in \cite{hyperbolic-dim3}, cf. in particular Example 5.11.

\subsection{Exceptional group $\op{G}_2$}

The stably free vector bundles of Mohan Kumar can also be used to provide examples of $\op{G}_2$-torsors which become trivial after extending the structure group along the homomorphism $\op{G}_2\to\op{Spin}(7)$. 

The norm form of the split octonion algebra gives rise to a homomorphism $\op{G}_2\to\op{Spin}(7)$, and the quotient can be identified as $\op{Spin}(7)/\op{G}_2\cong\op{Q}_7$. Consequently, there is an  $\mathbb{A}^1$-fiber sequence 
\[
\op{Q}_7\to{\op{B}}_{\op{Nis}}\op{G}_2\to{\op{B}}_{\op{Nis}}\op{Spin}(7).
\]
Mohan Kumar's construction provides a morphism $Y\cap Z\to \op{Q}_7$, and the composition with the map $\op{Q}_7\to{\op{B}}_{\op{Nis}}\op{G}_2$ classifies a $\op{G}_2$-torsor whose associated $\op{Spin}(7)$-torsor is trivial. 

A more explicit description of the torsor can be given as follows. The Zorn vector matrices provide a construction of $\op{G}_2$-torsors from oriented rank 3 vector bundles. If $k$ is an algebraically closed field, Mohan Kumar's construction provides oriented rank 3 vector bundles over a 4-dimensional smooth affine variety over $k(T)$, or by clearing denominators over a 5-dimensional smooth affine variety over $k$. Applying the Zorn vector matrix construction to these vector bundles provides $\op{G}_2$-torsors over the varieties constructed by Mohan Kumar. 

As discussed in \cite{octonion}, $\op{H}^3_{\op{Nis}}(X,\bm{\pi}^{\mathbb{A}^1}_3{\op{B}}_{\op{Nis}}\op{G}_2)\cong \op{H}^3_{\op{Nis}}(X,\mathbf{K}^{\op{M}}_4/3)$. Since Mohan Kumar's vector bundles are detected by non-trivial classes in the latter cohomology group, the $\op{G}_2$-torsor described above is in fact non-trivial. This provides examples of octonion algebras with trivial norm form over smooth affine varieties of smallest possible dimension. Full details for the proofs of the above assertions can be found in \cite[Section 4.3]{octonion}.

\end{document}